\documentclass[11pt,a4paper]{article}
\usepackage{amssymb,amsmath,amscd,amsthm,amsxtra,latexsym}
\theoremstyle{plain}
\newtheorem{theorem}{Theorem}[section]
\newtheorem{proposition}[theorem]{Proposition}
\newtheorem{lemma}[theorem]{Lemma}
\newtheorem{corollary}[theorem]{Corollary}
\newtheorem{definition}[theorem]{Definition}

\newtheorem{example}[theorem]{Example}

\begin{document}
\title{Extrinsic homogeneity of parallel submanifolds}
\author{Tillmann Jentsch}
\date{\today}
\maketitle

\sloppy
\renewcommand{\thefootnote}{\fnsymbol{footnote}}
\renewcommand{\labelenumi}{(\alph{enumi})}
\renewcommand{\labelenumii}{(\roman{enumii})}
\renewcommand{\labelenumiii}{(\arabic{enumiii})}


\newcommand{\R}{\mathrm{I\!R}}
\newcommand{\N}{\mathrm{I\!N}}
\newcommand{\HH}{\mathrm{I\!H}}
\newcommand{\F}{\mathrm{I\!F}}
\newcommand{\E}{\mathrm{I\!E}}
\newcommand{\K}{\mathrm{I\!K}}
\newcommand{\PP}{\mathrm{I\!P}}
\newcommand{\Z}{\mathbb{Z}}
\newcommand{\Q}{\mathbb{Q}}
\newcommand{\C}{\mathbb{C}}
\newcommand{\OO}{\mathbb{O}}

\newcommand{\scrA}{\mathcal{A}}
\newcommand{\scrB}{\mathcal{B}}
\newcommand{\scrC}{\mathcal{C}}
\newcommand{\scrD}{\mathcal{D}}
\newcommand{\scrE}{\mathcal{E}}
\newcommand{\scrF}{\mathcal{F}}
\newcommand{\scrG}{\mathcal{G}}
\newcommand{\scrH}{\mathcal{H}} 
\newcommand{\scrI}{\mathcal{I}}
\newcommand{\scrJ}{\mathcal{J}}
\newcommand{\scrK}{\mathcal{K}}
\newcommand{\scrL}{\mathcal{L}}
\newcommand{\scrM}{\mathcal{M}}
\newcommand{\scrN}{\mathcal{N}}
\newcommand{\scrO}{\mathcal{O}}
\newcommand{\scrP}{\mathcal{P}}
\newcommand{\scrQ}{\mathcal{Q}}
\newcommand{\scrR}{\mathcal{R}}
\newcommand{\scrS}{\mathcal{S}}
\newcommand{\scrT}{\mathcal{T}}
\newcommand{\scrU}{\mathcal{U}}
\newcommand{\scrV}{\mathcal{V}} 
\newcommand{\scrW}{\mathcal{W}}
\newcommand{\scrX}{\mathcal{X}}
\newcommand{\scrY}{\mathcal{Y}}
\newcommand{\scrZ}{\mathcal{Z}}

\newcommand{\frakA}{\mathfrak{A}}
\newcommand{\frakB}{\mathfrak{B}}
\newcommand{\frakC}{\mathfrak{C}}
\newcommand{\frakD}{\mathfrak{D}}
\newcommand{\frakE}{\mathfrak{E}}
\newcommand{\frakF}{\mathfrak{F}}
\newcommand{\frakG}{\mathfrak{G}}
\newcommand{\frakH}{\mathfrak{H}}
\newcommand{\frakI}{\mathfrak{I}}
\newcommand{\frakJ}{\mathfrak{J}}
\newcommand{\frakK}{\mathfrak{K}}
\newcommand{\frakL}{\mathfrak{L}}
\newcommand{\frakM}{\mathfrak{M}}
\newcommand{\frakN}{\mathfrak{N}}
\newcommand{\frakO}{\mathfrak{O}}
\newcommand{\frakP}{\mathfrak{P}}
\newcommand{\frakQ}{\mathfrak{Q}}
\newcommand{\frakR}{\mathfrak{R}}
\newcommand{\frakS}{\mathfrak{S}}
\newcommand{\frakT}{\mathfrak{T}}
\newcommand{\frakU}{\mathfrak{U}}
\newcommand{\frakV}{\mathfrak{V}}
\newcommand{\frakW}{\mathfrak{W}}
\newcommand{\frakX}{\mathfrak{X}}
\newcommand{\frakY}{\mathfrak{Y}}
\newcommand{\frakZ}{\mathfrak{Z}}

\newcommand{\fraka}{\mathfrak{a}}
\newcommand{\frakb}{\mathfrak{b}}
\newcommand{\frakc}{\mathfrak{c}}
\newcommand{\frakd}{\mathfrak{d}}
\newcommand{\frake}{\mathfrak{e}}
\newcommand{\frakf}{\mathfrak{f}}
\newcommand{\frakg}{\mathfrak{g}}
\newcommand{\frakh}{\mathfrak{h}}
\newcommand{\fraki}{\mathfrak{i}}
\newcommand{\frakj}{\mathfrak{j}}
\newcommand{\frakk}{\mathfrak{k}}
\newcommand{\frakl}{\mathfrak{l}}
\newcommand{\frakm}{\mathfrak{m}}
\newcommand{\frakn}{\mathfrak{n}}
\newcommand{\frako}{\mathfrak{o}}
\newcommand{\frakp}{\mathfrak{p}}
\newcommand{\frakq}{\mathfrak{q}}
\newcommand{\frakr}{\mathfrak{r}}
\newcommand{\fraks}{\mathfrak{s}}
\newcommand{\frakt}{\mathfrak{t}}
\newcommand{\fraku}{\mathfrak{u}}
\newcommand{\frakv}{\mathfrak{v}}
\newcommand{\frakw}{\mathfrak{w}}
\newcommand{\frakx}{\mathfrak{x}}
\newcommand{\fraky}{\mathfrak{y}}
\newcommand{\frakz}{\mathfrak{z}}

\newcommand{\rmA}{\mathrm{A}}
\newcommand{\rmB}{\mathrm{B}}
\newcommand{\rmC}{\mathrm{C}}
\newcommand{\rmD}{\mathrm{D}}
\newcommand{\rmE}{\mathrm{E}}
\newcommand{\rmF}{\mathrm{F}}
\newcommand{\rmG}{\mathrm{G}}
\newcommand{\rmH}{\mathrm{H}}
\newcommand{\rmI}{\mathrm{I}}
\newcommand{\rmJ}{\mathrm{J}}
\newcommand{\rmK}{\mathrm{K}}
\newcommand{\rmL}{\mathrm{L}}
\newcommand{\rmM}{\mathrm{M}}
\newcommand{\rmN}{\mathrm{N}}
\newcommand{\rmO}{\mathrm{O}}
\newcommand{\rmP}{\mathrm{P}}
\newcommand{\rmQ}{\mathrm{Q}}
\newcommand{\rmR}{\mathrm{R}}
\newcommand{\rmS}{\mathrm{S}}
\newcommand{\rmT}{\mathrm{T}}
\newcommand{\rmU}{\mathrm{U}}
\newcommand{\rmV}{\mathrm{V}}
\newcommand{\rmW}{\mathrm{W}}
\newcommand{\rmX}{\mathrm{X}}
\newcommand{\rmY}{\mathrm{Y}}
\newcommand{\rmZ}{\mathrm{Z}}

\newcommand{\bbA}{\mathbb{A}}
\newcommand{\bbB}{\mathbb{B}}
\newcommand{\bbC}{\mathbb{C}}
\newcommand{\bbD}{\mathbb{D}}
\newcommand{\bbE}{\mathbb{E}}
\newcommand{\bbF}{\mathbb{F}}
\newcommand{\bbG}{\mathbb{G}}
\newcommand{\bbH}{\mathbb{H}}
\newcommand{\bbI}{\mathbb{I}}
\newcommand{\bbJ}{\mathbb{J}}
\newcommand{\bbK}{\mathbb{K}}
\newcommand{\bbL}{\mathbb{L}}
\newcommand{\bbM}{\mathbb{M}}
\newcommand{\bbN}{\mathbb{N}}
\newcommand{\bbO}{\mathbb{O}}
\newcommand{\bbP}{\mathbb{P}}
\newcommand{\bbQ}{\mathbb{Q}}
\newcommand{\bbR}{\mathbb{R}}
\newcommand{\bbS}{\mathbb{S}}
\newcommand{\bbT}{\mathbb{T}}
\newcommand{\bbU}{\mathbb{U}}
\newcommand{\bbV}{\mathbb{V}}
\newcommand{\bbW}{\mathbb{W}}
\newcommand{\bbX}{\mathbb{X}}
\newcommand{\bbY}{\mathbb{Y}}
\newcommand{\bbZ}{\mathbb{Z}}

\newcounter{neqcounter}[section]
\newcounter{seqcounter}[section]
\newcommand{\numberequation}[1]
   {\renewcommand{\theequation}{\arabic{equation}}
    \setcounter{equation}{\value{neqcounter}}
    \begin{equation}#1\end{equation}\stepcounter{neqcounter}}
\newcommand{\symbolequation}[1]
   {\renewcommand{\theequation}{\fnsymbol{equation}}
    \setcounter{equation}{\value{seqcounter}}
    \begin{equation}#1\end{equation}\stepcounter{seqcounter}}

\newcommand{\diff}{\mathrm{d}}
\newcommand{\Diff}{\mathrm{D}}
\newcommand{\id}{\mathrm{id}}
\newcommand{\GL}{\mathrm{GL}}
\newcommand{\SL}{\mathrm{SL}}
\newcommand{\SO}{\mathrm{SO}}
\newcommand{\End}{\mathrm{End}}
\newcommand{\Alt}{\mathrm{Alt}}
\newcommand{\Spur}{\mathrm{Spur}}
\newcommand{\rk}{\mathrm{rk}}
\newcommand{\Kern}{\mathrm{Kern}}
\newcommand{\Spann}{\mathrm{Spann}}
\newcommand{\Top}{\mathrm{Top}}
\newcommand{\del}{\partial}
\newcommand{\ddel}[2]{\frac{\partial#1}{\partial#2}}
\newcommand{\ddeldel}[3]{\frac{\partial^2#1}{\partial#2\:\partial#3}}
\newcommand{\ddelsquare}[2]{\frac{\partial^2#1}{\partial{#2}^2}}
\newcommand{\grad}{\mathrm{grad}}
\newcommand{\hess}{\mathrm{hess}}
\newcommand{\Hess}{\mathrm{Hess}}
\newcommand{\Rp}{\R_+}
\newcommand{\eps}{\varepsilon}
\newcommand{\vi}{\varphi}
\newcommand{\vkap}{\varkappa} 
\newcommand{\kap}{\varkappa}  
\newcommand{\thet}{\vartheta}
\newcommand{\ro}{\varrho}
\newcommand{\Hol}{\mathrm{Hol}}
\newcommand{\rg}{\mathrm{rg}}

\renewcommand{\O}{\varnothing}

\newcommand{\UL}{\mathchoice
{\setbox0=\hbox{$\displaystyle U$}
 \setbox1=\hbox{$\displaystyle l$}
 \hbox{\box0\kern-0.80\wd1\lower0.029\ht1\box1}}
{\setbox0=\hbox{$\displaystyle U$}
 \setbox1=\hbox{$\displaystyle l$}
 \hbox{\box0\kern-0.80\wd1\lower0.029\ht1\box1}}
{\setbox0=\hbox{$\scriptstyle U$}
 \setbox1=\hbox{$\scriptstyle l$}
 \hbox{\box0\kern-0.80\wd1\lower0.029\ht1\box1}}
{\setbox0=\hbox{$\scriptscriptstyle U$}
 \setbox1=\hbox{$\scriptscriptstyle l$}
 \hbox{\box0\kern-0.80\wd1\lower0.029\ht1\box1}}
}

\newcommand{\Uo}{\UL^o}
\def\fam#1#2#3#4{({#1}_{#2})_{#2=#3,\dotsc,#4}}
\def\zz#1#2#3{#1=#2,\dotsc,#3}
\newcommand{\intint}[2]{\{#1,\dotsc,#2\}}
\newcommand{\Vektor}[2]{({#1}_1,\dotsc,{#1}_{#2})} 

\newcommand{\qmq}[1]{\quad\mbox{#1}\quad}
\newcommand{\Menge}[2]{\{\,#1\,|\,#2\,\}}
\newcommand{\Abstand}[1]{\mbox{}\par\vspace{-#1mm}}
\newcommand{\ninfty}{{-}\infty}

\def\bild#1#2#3#4{\leavevmode\vbox to #1{\vfil \hbox to #2{\special{picture #4
   scaled #3}\hfil}}}

\newcommand{\normRahmen}[1]
           {\setlength{\fboxsep}{8pt}\fbox{#1}}

\newcommand{\grRahmen}[1]{\begin{center}\setlength{\fboxrule}{0.8pt}
           \setlength{\fboxsep}{8pt}
              \fbox{\begin{minipage}{140mm}\rule{0mm}{5mm}\hspace*{-2mm} #1
                    \end{minipage}}\end{center}}

\newcommand{\klRahmen}[1]{\begin{center}\large\setlength{\fboxrule}{0.8pt}
                     \setlength{\fboxsep}{8pt}
                     \fbox{\rule[-2mm]{0mm}{7mm} #1 }
                     \normalsize\end{center}}

\newcommand{\varRahmen}[2]{\begin{center}\setlength{\fboxrule}{0.8pt}
           \setlength{\fboxsep}{8pt}
              \fbox{\begin{minipage}{#1}\rule{0mm}{5mm}\hspace*{-2mm} #2
                    \end{minipage}}\end{center}}


\setlength{\marginparsep}{5pt}
\setlength{\marginparwidth}{30pt}
\newcommand{\marginlabel}[1]           
                 {\mbox{}\marginpar{\raggedleft\hspace{0pt}#1}}
\newcommand{\Ausrufezeichen}{\marginlabel{\raisebox{-1.2ex}{\Huge$\boldsymbol{!}$}}}
\newcommand{\Randbalken}[2]{\marginlabel{{\rule[#1]{0.5mm}{#2}}}}
\newcommand{\Zeigefinger}{\marginlabel{\raisebox{-0.8ex}{\LARGE\ding{43}}}}
\newcommand{\Blume}{\marginlabel{\raisebox{-0.6ex}{\LARGE\ding{94}}}}

\newcommand{\bigoperp}{\mathop{\bigcirc\raisebox{0.25em}
 {\hskip-0.53em\hbox{\vrule height1.0ex width0.04em}
  \hskip-0.77em\hbox{\vrule height0.04em width 0.8em}
  \hskip- 0.2em}}}

\renewcommand{\bigoplus}{\mathop{\bigcirc
  \raisebox{-0.22em}{\hskip-0.53em\hbox{\vrule height2.08ex width0.04em}
  \raisebox{ 0.48em}{\hskip-0.75em\hbox{\vrule height0.04em width 0.8em}}
  \hskip- 0.2em}}}

\def\NP#1#2#3{\mathchoice
  {{\textstyle{\prod\limits_{i=#2}^{#3}}{#1}_i}} 
  {{\textstyle{\prod_{i=#2}^{#3}}{#1}_i}}
  {}
  {}
   }
\def\DS#1#2#3{\bigoplus_{i=#2}^{#3}\!#1_i} 
\def\OS#1#2#3{\bigoperp_{i=#2}^{#3}#1_i} 
\def\WP#1#2#3{#1_0 \times_{#2}\NP {#1}{1}{#3}}
\def\TP#1#2#3#4{{\rule{0mm}{2ex}}^{#2}\NP{#1}{#3}{#4}}

\newcommand {\cinf}{\ensuremath{\mathrm{C}^{\infty}}}
\newcommand {\X}{\mathfrak{X}}
\newcommand{\Tensor}[3]{\mathfrak{T}^{(#2,#3)}(#1)}
\newcommand{\alphad}{\dot{\alpha}}
\newcommand {\g}[2]{\langle #1,#2\rangle}
\def\gg{\g{\cdot\,}{\cdot}}
\newcommand {\euc}[2]{\langle\!\langle #1,#2 \rangle\!\rangle}
\newcommand {\peuc}[2]{\langle\!\langle #1,#2 \rangle\!\rangle\!_s}
\newcommand {\Kov}[2]{\nabla_{#1}#2}
\newcommand {\Kovh}[2]{\widehat{\nabla}_{#1}#2}
\newcommand {\Kovt}[2]{\widetilde{\nabla}_{#1}#2}
\newcommand {\Kovperp}[2]{\nabla^{\perp}_{#1}#2}
\newcommand {\Kovv}[3]{{\nabla^{\scriptscriptstyle{#1}}}_{\!#2}#3}
\newcommand{\Kodel}[1]{\Kov{\del}{#1}}
\newcommand{\Kokodel}[1]{\Kov{\del}{\Kov{\del}{#1}}}

\newcommand {\operp}{\mathbin{\mbox{$\ominus\raisebox{2.9pt}
 {\hskip-0.42em\hbox{\vrule height0.7ex width0.02em}\hskip0.42em }$}}}
\newcommand {\TpM}{T_{p}M}
\newcommand {\Tpf}{T_{p}f}
\newcommand {\Nf}{\perp\!\!(f)}
\newcommand {\NM}{\perp\!\!M}
\newcommand {\Npf}{\perp_p\!\!(f)}
\newcommand {\NpM}{\perp_p\!\!M}
\newcommand {\Nepf}{\perp^{\!\!1}_p\!\!(f)}
\newcommand {\NepM}{\perp^{\!\!1}_p\!\!M}
\newcommand{\Shop}[3]{{\mathrm{A}^{^{\scriptscriptstyle{\!\!#1}}}_{#2}}#3} 

\def\tscrV{\widetilde{\mathcal{V}}}
\def\tscrH{\widetilde{\mathcal{H}}}
\newcommand{\hdisp}[3]{\overset{#2}{\underset{#1}{\parallel}}\!\!#3\,} 

\newcommand{\dotdiff}[2]{\dot{\frac{\diff\ }{\diff #1}}\Big|_{#1=0}\big(#2\big)}
\newcommand{\dotdif}[2]{\dot{\frac{\diff\ }{\diff #1}}\Big|_{#1=0}#2}

\newcommand{\normaldiff}[3]{\frac{\diff\ }{\diff #1}\Big|_{#1=#2}\big(#3\big)}
\newcommand{\normaldif}[3]{\frac{\diff\ }{\diff #1}\Big|_{#1=#2}#3}

\newcommand {\RPn}{\ensuremath{\R\mathrm{P}^{n}}}
\newcommand {\RPm}{\ensuremath{\R\mathrm{P}^{m}}}

\newcommand {\CPn}{\ensuremath{\C\mathrm{P}^{n}}}
\newcommand {\CPm}{\ensuremath{\C\mathrm{P}^{m}}}
\newcommand {\CHn}{\ensuremath{\C\mathrm{H}^{n}}}
\newcommand {\CHm}{\ensuremath{\C\mathrm{H}^{m}}}

\hyphenation{Riemann-sche Riemann-ian}


\newcommand{\sst}{\scriptscriptstyle}
\newcommand{\VV}{\mathbb{V}}


\newcommand{\PV}{\rmP(\bbV)}
\newcommand{\SV}{\rmS(\bbV)}
\newcommand{\EPV}{\E\times\PV}
\newcommand{\ESV}{\E\times\SV}


\newcommand{\GrTM}{\rmG_r(TM)}  
\newcommand{\Kovp}[1]{\Kov{}{#1}\big|_{p_0}} 

\newcommand{\GrnTM}{\rmG_r^\nabla(TM)}
\newcommand{\MtM}{{M^\times}}

%

\newcommand{\ghdisp}[4]{(\hdisp{#1}{#2}{#3})^{\sst{#4}}}
\newcommand{\hh}[1]{\hat{\!\hat{#1}}} 
\renewcommand{\R}{\bbR}
\renewcommand{\Kodel}[2]{\nabla^
  {\text{\raisebox{0.5ex}{${\scriptstyle{#1}}$}}}_{\del}{#2}}
\renewcommand{\Kokodel}[2]{\Kodel{#1}{\Kodel{#1}{#2}}}
\renewcommand{\E}{\bbE}
\renewcommand{\Kovh}[2]{\hat\nabla_{#1}{#2}}
\newcommand{\WBKov}[2]{\overline{\nabla}_{#1}{#2}}
\newcommand{\Till}{\ensuremath {\scrT}}
\newcommand{\Tillb}{\ensuremath {\Till_b}}
\newcommand{\Tsukada}{\ensuremath {\scrD}}
\newcommand{\Tsukadaq}{\ensuremath {\Tsukada_q}}
\newcommand{\Gzwei}{\ensuremath {\rmG_m^{2}(N)}}
\newcommand{\Isom}{(\hat\eta,\hat\nu)}
\newcommand{\hv}{\hat{v}}
\newcommand{\hu}{\hat{u}}
\newcommand{\hN}{\hat{N}}
\newcommand{\hnu}{\hat\nu}
\newcommand{\hW}{\hat{W}}
\newcommand{\hX}{\hat{X}}
\newcommand{\hY}{\hat{Y}}
\newcommand{\hx}{\hat{x}}
\newcommand{\hy}{\hat{y}}
\newcommand{\hU}{\hat{U}}
\newcommand{\hV}{\hat{V}}
\newcommand{\hZ}{\hat{Z}}
\newcommand{\hOmega}{{\hat{\Omega}}}

\newcommand{\GamA}{\Gamma_{\!\!A}}
\newcommand{\piP}{\pi^{\sst\bbP}}
\newcommand{\rmFB}{\mathrm{FB}}
\newcommand{\Besym}{\scrB_{esym}(N)}
\newcommand{\Besymo}{\scrB_{esym}^0(N)}
\newcommand{\ttop}{{\boldsymbol{\top\hspace{-0.75em}\top}}}
\newcommand{\bbot}{{\boldsymbol{\!\bot\hspace{-0.75em}\bot}}}

\newcommand{\kernel}{\mathrm{kernel}}

\newcommand{\Gl}{\mathrm{Gl}}
\newcommand{\SU}{\mathrm{SU}}
\newcommand{\Sp}{\mathrm{Sp}}
\newcommand{\so}{\mathfrak{so}}
\newcommand{\su}{\mathfrak{su}}
\newcommand{\gl}{\mathfrak{gl}}
\renewcommand{\sp}{\mathfrak{sp}}
\newcommand{\tr}{\mathfrak{tr}}
\newcommand{\osc}{\scrO}
\newcommand{\Id}{\mathrm{Id}}
\newcommand{\Iso}{\mathrm{I}}
\newcommand{\trace}{\mathrm{trace}}
\newcommand{\diag}{\mathrm{diag}}
\newcommand{\Sym}{\mathrm{Sym}}
\renewcommand{\i}{\mathrm{i}}
\renewcommand{\j}{\mathrm{j}}
\renewcommand{\k}{\mathrm{k}}
\renewcommand{\Spann}[2]{\{#1\big|#2\}_{\scriptstyle\R} }\,
\newcommand{\hol}{\mathfrak{hol}}
\newcommand{\fetth}{\boldsymbol{h}}
\newcommand{\fettb}{\boldsymbol{b}}
\newcommand{\ad}{\mathrm{ad}}
\newcommand{\Ad}{\mathrm{Ad}}
\newcommand{\Hom}{\mathrm{Hom}}
\newcommand{\codim}{\mathrm{codim}}
\newcommand{\rank}{\mathrm{rank}}
\newcommand{\Tr}{\mathrm{Tr}}

\abstract{We consider parallel submanifolds $M$ of a Riemannian
symmetric space $N$ and study the question whether 
$M$ is extrinsically homogeneous in $N$\,, i.e.\ whether there exists a subgroup of the isometry group of
$N$ which acts transitively on $M$\,. First, given a ``2-jet'' $(W,b)$ at
some point $p\in N$ (i.e. $W\subset T_pN$ is a linear space and $b:W\times
W\to W^\bot$ is a symmetric bilinear form)\,, we derive necessary and sufficient
conditions for the existence of a parallel submanifold with extrinsically homogeneous tangent holonomy
bundle which passes through $p$ and whose 2-jet at $p$ is
given by $(W,b)$\,. Second, we focus our attention on complete,
(intrinsically) {\em irreducible} \/parallel submanifolds of $N$\,. 
Provided that $N$ is of compact or non-compact type, we establish the
extrinsic homogeneity of every complete, irreducible parallel submanifold
of $N$ whose dimension is at least 3 and which is not contained in any flat of $N$\,.}
\section{Introduction}
\label{se:introduction}
Given a Riemannian manifold $M$ and an 
isometric immersion $f:M\to N$ into some Riemannian symmetric space
$N$\,, we let $TM$ denote the tangent bundle
of $M$\,, $\bot f$ the normal bundle of $f$\,, $h:TM\times TM\to\bot f$ the second
fundamental form and $S:TM\times\bot f\to TM,(x,\xi)\mapsto S_\xi(x)$ the
shape operator. $\nabla^M$ and $\nabla^N$ denote the
Levi Civita connection of $M$ and $N$\,, respectively, and $\nabla^\bot$
denotes the usual normal connection on $\bot f$ (obtained by projection).
Then there is the splitting $f^*TN=TM\oplus \bot f$ and 
the equations of Gau{\ss} and Weingarten state that for $X,Y\in \Gamma(TM), \xi\in\Gamma(\bot f)$
\begin{align*}
&\Kovv{N}{X}{Tf\,Y}=Tf(\Kovv{M}{X}{Y})+h(X,Y)\;,\\
&\Kovv{N}{X}{\xi}=-Tf(S_\xi(X))+\Kovv{\bot}{X}{\xi}\;.
\end{align*}
Furthermore, $h_p$ is a symmetric bilinear map on $T_pM$ with values in $\bot_pf$ for each $p\in M$\,, i.e. 
$h$ is a section of the vector bundle $\Sym^2(TM,\bot f)$\,. On the latter vector bundle, there is a connection induced by
$\nabla^M$ and $\nabla^\bot$ in the usual way, often called ``Van der Waerden-Bortolotti connection''.

\begin{definition}\label{de:parallel} $f$ is called
 \emph{parallel} if $h$ is a parallel section.
\end{definition}
In a similar fashion, we define parallel submanifolds of $N$ (via the isometric immersion
defined by the inclusion map $\iota^M:M\hookrightarrow N$). It is known that a parallel submanifold of $N$ is 
always intrinsically a locally symmetric space and that a simply connected complete Riemannian manifold $M$ is necessarily a globally symmetric space
in case there exists a parallel isometric immersion $f:M\to N$ (since $N$ is a symmetric space).

\begin{example}
1-dimensional parallel isometric immersions $c:\R\to N$ are either
geodesics or {\em (extrinsic) circles}\/ (in the sense of~\cite{NY});
they are given by the Frenet curves of osculating rank one and two, resp., parameterized by arc-length.
\end{example}

Let $\Iso(N)$ denote the Lie
group of isometries of $N$ (see~\cite[Ch.~IV,~\S~2 and~\S~3]{He}).
Given a connected Lie subgroup $G\subset \Iso(N)$ and some $p\in M$\,, let
$H\subset G$ denote the isotropy subgroup at $p$\,. Then the orbit
$M:=G\,p\cong G/H$ is a submanifold of $N$
(cf.~\cite[Ch.~2.9]{Var}) called an {\em (extrinsically) homogeneous submanifold}\,;
however, the topology of $M$ is not necessarily the induced topology, i.e. $M$ is not necessarily  a {\em regular} \/submanifold.\footnote{\label{Fn:quasi}
One can show that $M$ is a {\em quasi-regular}\/ submanifold of $N$ (cf.~\cite[p.~17/18]{Var})\,, which
means the following: If $P$ is any manifold and $u$ is any (not necessarily
continuous) map defined from $P$ to $M$\,, then $u$ is differentiable if and only if
$\iota^M\circ u$ is differentiable. The proof of this result can be found
in~\cite[Vol.~II, Ch.~III, \S~2, Theorem~1]{GHV}}

Since parallelity of $h$ can be seen as the extrinsic analogue of the
characterization of a locally symmetric space, $\nabla R=0$\,, one should
intuitively expect that a complete parallel submanifold of $N$ is extrinsically homogeneous.
In fact, if $N$ is a Euclidian space and $M$ is a complete parallel submanifold of $N$\,,
then it follows from a result of~\cite{Fe1} that $M$ is a {\em symmetric submanifold}\/
(i.e. $M$ is invariant under the reflections at the various normal
spaces, see~\cite{Fe2}); another proof of his observation was given in~\cite{St}\,.
Clearly, every symmetric submanifold of an arbitrary ambient space (in the
sense of~\cite{N2,St}) is a homogeneous parallel submanifold, but there
always exist parallel submanifolds which are not extrinsically symmetric
unless the ambient space has constant curvature. 

Nevertheless, if the ambient space $N$ is a rank-1 symmetric space, then every
complete parallel submanifold of $N$ turns out to be
extrinsically homogeneous (for curves, this follows from~\cite[Theorem~2.1]{MT}; for higher-dimensional submanifolds, this is a
consequence of the classification of parallel submanifolds in rank-1 symmetric spaces, see~\cite[Ch.~9.3]{BCO}).
However, this fact remains no longer true if $N$ is of higher rank: In
fact, according to~\cite{MT}, a homogeneous space $N$ is Euclidian
or a symmetric space of rank 1 if and only if every circle $c:\R¸\to N$ is the orbit of a one-parameter
subgroup of $\Iso(N)$\,.

In this article, we will study extrinsic homogeneity of parallel
submanifolds of $N$\,; thereby the case that $N$ is of higher rank is always
implicitly included. In this case there seems to be ``not
much known about parallel submanifolds of $N$'' (cited from~\cite{BCO})\,;
hence our results might serve as a first step towards a better
understanding of parallel submanifolds in ambient symmetric spaces of higher
rank.

\subsection{The main result}
\label{se:main}
If $M$ is a Riemannian manifold, $p\in M$ is a fixed point and $e:=(e_1,\ldots
e_m)$ is an orthonormal basis of $T_pM$\,, then, by definition, the holonomy bundle of $TM$
through $e$ (denoted by $\Hol(M)$) is obtained by
parallel translation of $e$ along the various curves
$[0,1]\to M$ with $c(0)=p$\,. Clearly, every subgroup $G\subset\Iso(M)$
acts on the principal bundle of orthonormal frames over $M$ in a natural way; but, in general, $G$ does not
leave $\Hol(M)$ invariant. However, if $M$ is a symmetric space,
then the subgroup of $\Iso(M)$ which is generated by the transvections (see
Appendix~\ref{se:C}) acts transitively on $\Hol(M)$\,.

\begin{definition}\label{de:homogeneous_holonomy}
Let $M$ be a (quasiregular) submanifold of the symmetric space $N$\,.
$M$ has extrinsically homogeneous (tangent) holonomy
bundle if there exists a Lie subgroup $G\subset\Iso(N)$ 
which leaves $M$ invariant and some $p\in M$ such that for every curve $c:[0,1]\to M$ with $c(0)=p$ there exists some $g\in G$ with
\begin{itemize}
\item $g(p)=c(1)$\,, 
\item and the parallel displacement along $c$ is given by
\begin{equation}\label{eq:f_is_equivariant}
\ghdisp{0}{1}{c}{M}=T_pg|_{T_pM}:T_pM\to T_{c(1)}M\;.
\end{equation}
\end{itemize}
\end{definition}
Clearly, then $M$ is the orbit $G\,p$\,.

\begin{example}\label{ex:homogeneous_holonomy}
\begin{enumerate}
\item
Every symmetric submanifold of $N$ has extrinsically
homogeneous	 holonomy bundle.
\item For every curve $c:\R\to N$ the following assertions are equivalent:
\begin{itemize}
\item $c$ is the orbit of a one-parameter subgroup of $\Iso(N)$\;.
\item $c(\R)$ is a homogeneous submanifold.
\item $c(\R)$ is a submanifold with extrinsically homogeneous holonomy bundle.
\end{itemize}
\end{enumerate}
\end{example}

In the following, an intrinsically flat, totally geodesic submanifold of the symmetric space $N$ is shortly
called a flat. Recall that $N$ is called of compact or non-compact type if the
Killing form of $\fraki(N)$ restricted to $\frakp$ is strictly negative
or strictly positive, respectively (see~\cite[Ch.~V,~\S~1]{He}), and that a
simply connected symmetric space $M$ is called reducible if it is isometric to the Riemannian product of two symmetric spaces
of dimension at least one, respectively (otherwise $M$ is called irreducible).

\begin{theorem}[Main Result]\label{th:MainResult}
Let $f$ be a parallel isometric immersion from a simply connected irreducible
symmetric space $M$ with $\dim(M)\geq 3$ into a symmetric space $N$
of compact or non-compact type. Then the following assertions are equivalent:
\begin{enumerate}
\item $f(M)$ is not contained in any flat of $N$\,.
\item $f(M)$ is a homogeneous submanifold of $N$\,.
\item $f(M)$ is a submanifold with extrinsically homogeneous holonomy bundle.
\end{enumerate}
If any of these properties holds, then $f(M)$ is actually
a parallel submanifold of $N$ and $f:M\to f(M)$ is a Riemannian covering.
\end{theorem}

Recall that an isometric immersion
$f:M\to N$ is called full if $f(M)$ is not contained in any proper totally
geodesic submanifold of $N$ (cf.~\cite[Ch.~2.5]{BCO}) and
that the rank of $N$ is the dimension of a maximal flat in $N$ (see~\cite[Ch.~V,~\S~6]{He})\,.

\begin{corollary} In the situation of Theorem~\ref{th:MainResult}, suppose that
$\dim(M)$ is greater than the rank of $N$ or that $f$ is full.
Then $f(M)$ is a parallel submanifold with extrinsically homogeneous
holonomy bundle and $f:M\to f(M)$ is a Riemannian covering.
\end{corollary}

It is known that every homogeneous submanifold is complete and that the universal covering space of every complete parallel
submanifold of $M$ is a symmetric space. Therefore:

\begin{corollary}Let $N$ be a symmetric space of compact or non-compact
  type. Then every locally irredcuible homogeneous parallel submanifold of $N$
  of dimension at least two has extrinsically homogeneous holonomy bundle.
\end{corollary}

Furthermore, according
to~\cite[Theorem~7]{JR}, for every (possibly not complete) parallel
submanifold $M_{loc}\subset N$ there exists a simply connected symmetric space
$M$\,, an open subset $U\subset M$ and a parallel isometric immersion $f:M\to N$
such that $f|_U:U\to M_{loc}$ is a covering. Hence:

\begin{corollary}
Let $N$ be a symmetric space of compact or non-compact type and $M_{loc}$
be a (possibly not complete) locally irreducible parallel submanifold of $N$ with $\dim(M_{loc})\geq 3$\,.
Suppose that $M_{loc}$ is not contained in any flat of $N$\,. Then $M_{loc}$
is an open part of a submanifold with extrinsically homogeneous holonomy bundle.
\end{corollary}

\sloppypar
\begin{example}\label{ex:extrinsic_holonomy}
Consider the Riemannian product space $N:=\C\rmP^1\times\C\rmP^1$\,, which is a
Hermitian symmetric space of compact type and rank $2$\,.
Let $c:\R\to N$ be a ``generic'' circle (this condition will be explained more
precisely in the proof of this example). Then $c$ is a full parallel isometric
immersion, but $c(\R)$ is not a homogeneous submanifold of $N$\,.
\end{example}

\paragraph{This article is organized as follows:} 
In Section~\ref{se:2}, we recall some well known results on parallel submanifolds of symmetric spaces.

In Section~\ref{se:3}, we consider ``infinitesimal models'' of parallel submanifolds with extrinsically homogeneous holonomy bundle: 
given a formal 2-jet at some point $p\in N$\,, we ask the question whether this 2-jet is induced by a
parallel submanifold with extrinsically homogeneous holonomy bundle which passes through $p$
(see Definition~\ref{de:model} and Theorem~\ref{th:2-jet}). The proof of
Theorem~\ref{th:2-jet} 
can be found in Sections~\ref{se:3.1} and~\ref{se:3.2}. 
As a first application of Theorem~\ref{th:2-jet}, 
given a symmetric submanifold of some symmetric space $\tilde N$\,, 
there seems to be a possibility how to construct a full parallel submanifold
of the symmetric space $N$ which is not extrinsically symmetric (see
Theorem~\ref{th:new_examples?}). 
Furthermore, given a parallel isometric immersion $f:M\to N$\,, 
we can state necessarily and sufficient conditions on the 2-jet of 
$f$ at $p$ which decide whether $f(M)$ has extrinsically homogeneous holonomy bundle or not (Theorem~\ref{th:homogen}).

In Section~\ref{se:parallel_flat}, we deal with the proof of
Theorem~\ref{th:MainResult}. For a parallel isometric immersion $f$ from a simply connected irreducible
symmetric space $M$ into $N$\,, one can show that the associated totally
geodesic submanifold $\bar M:=\exp(T_pf(T_pM))$ is either a (locally)
irreducible symmetric space, too, or $\bar M$ is a flat (Proposition~\ref{p:hol-relations}). In the latter case,
provided also that $N$ is of compact or non-compact
type, we will show that even $f(M)$ is contained in some flat of $N$ (Theorem~\ref{th:4}); for
the proof of this result, we will use the classification of parallel submanifolds in the Euclidian
space given by D.~Ferus~\cite{Fe1,Fe2}. In the
first case, provided additionally that $\dim(M)\geq 3$\,, we can show that a
distinguished linear map $\fetth$ which encodes both the shape operator and the second fundamental form of
$f$ at $p$ (see~\eqref{eq:fetth1}) takes values in the ``extrinsic'' holonomy Lie algebra of $f$
(Theorem~\ref{th:fetth_in_hol}). With these preconsiderations, the proof of
Theorem~\ref{th:MainResult} is not very difficult anymore (see Section~\ref{se:proof}).

Finally, in Sec.~\ref{se:two-symmetric}, we discuss
``2-symmetry'' of parallel submanifolds (see Theorem~\ref{th:two-symmetric}).

\section{Some well known properties of parallel submanifolds of symmetric spaces}\label{se:2}
Let $N$ be a symmetric space and $f:M\to N$ be a parallel isometric immersion.
In the following, we implicitly identify $T_pM$ with $Tf(T_pM)$ (by means of $T_pf$) for each $p\in
M$\,; for convenience, the reader may assume that $M\subset N$ is a
submanifold and $f=\iota^M$\,. We introduce the {\em first normal space}
$\bot^1_pf:=\Spann{h(x,y)}{x,y\in T_pM}$ and the {\em second osculating space} $\osc_p
f:=T_pM\oplus \bot_p^1f$ for each $p\in M$\,. As $p$ varies over $M$\,, this defines the first normal bundle $\bot^1f$
and the (second) osculating bundle $\osc f=TM\oplus\bot^1 f$
(where $TM$ is seen as a vector subbundle of $f^*TN$ by means of $Tf$)\,.

Let $\sigma^\bot\in \rmO(\osc_p f)$ denote the linear reflection in $\bot^1_pf$ and
$\Ad(\sigma^\bot):\so(\osc_pf)\to\so(\osc_pf),A\mapsto \sigma^\bot\circ A\circ\sigma^\bot$
be the induced involution on $\so(\osc_pf)$\,. Let $\so(\osc_pf)_+$ and $\so(\osc_pf)_-$ denote the $+1$- and $-1$-eigenspaces of
$\Ad(\sigma^\bot)$\,, respectively, i.e.
\begin{align}
\label{eq:even}
&\so(\osc_pf)_+:=\left\{\left.\left(
\begin{array}{cc}
A & 0 \\
0 & B
\end{array}
\right )\right|A\in \so(W),B\in \so(\bot^1_pf\right\}\;, \\
\label{eq:odd}
&\so(\osc_pf)_-:=\left\{\left.\left(
\begin{array}{cc}
0 & -C^* \\
C & 0
\end{array}
\right)\right|C\in \Hom(W,\bot^1_p\right\}\;.
\end{align}

Then the following lemma is straightforward:

\begin{lemma}\label{le:splitting}
\begin{enumerate}
\item
We have $A\in\so(\osc_pf)_+$ if and only if $A(T_pM)\subset T_pM$\,.
\item
The map $\so(\osc_pf)_-\to\Hom(T_pM,\bot_pf),A\mapsto A|_{T_pM}$ is a linear isomorphism.
\end{enumerate}
\end{lemma}

Because of Lemma~\ref{le:splitting}~(b), there exists a unique
linear map $\fetth:T_pM\to\so(\osc_pf)_-$ characterized by
\begin{equation}\label{eq:fetth1}
\forall x,y\in W:\fetth(x)\,y=h(x,y)\;.
\end{equation}
Then $-\fetth:W\times \bot^1b\to W$ plays the role of the shape operator, i.e.
\begin{equation}\label{eq:fetth2}
\forall x,y\in W,\xi\in\bot^1_pf:\g{\fetth(x)\xi}{y}=-\g{h(x,y)}{\xi}\;.
\end{equation}

In the following, we use the natural inclusion
$\so(\osc_pf)\hookrightarrow\so(T_oN)$ ($o:=f(p)$) such that 
each $A\in\so(\osc_pf)$ becomes the zero-map on $\osc_pf^\bot$\,. 
Since $f$ is parallel, now the curvature equations of Gau{\ss}, Codazzi and Ricci
can formally be combined to
\begin{equation}\label{eq:Gauss_Ricci}
R^N(x,y)(z+\xi)=R^M(x,y)\,z+R^\bot(x,y)\,\xi +[\fetth(x),\fetth(y)](z+\xi)
\end{equation}
for all $x,y,z\in T_pM,\xi\in\bot_pM$\,.

\begin{proposition}\label{p:properties}
For every parallel isometric immersion $f:M\to N$ into the symmetric space
$N$\,, the following properties hold:
\begin{enumerate}
\item The tangent space $T_pM$ is a curvature invariant subspace of $T_oN$ ($o:=f(p)$)\,.
\item For arbitrary $p\in M$\,, $x,y\in T_pM$ we have
$\fetth(x)\,\fetth(y)(T_pM)\subset T_pM$ and the following equation holds on $\osc_pf$\,:
\begin{align}
\notag
R^N(h(x,x),h(y,y))=&[\fetth(x),[\fetth(y),R^N(x,y)]]-R^N(\fetth(x)\,\fetth(y)\,x,y)\\
\label{eq:fundamental}
&-R^N(x,\fetth(x)\,\fetth(y)\,y)\;.
\end{align}
\item
The first normal spaces $\bot^1_pf$ are curvature invariant subspaces of
$T_oN$\,, too.
\item The tensor of type $(0,4)$ on $\osc f$ defined by
$
R^\flat(v_1,v_2,v_3,v_4):=\g{R^N(v_1,v_2)\,v_3}{v_4}
$
satisfies
\begin{equation}\label{eq:fetth_3}
\sum_{i=1}^4R^{\flat}(v_1,\ldots,\fetth(x)\,v_i,\ldots,v_4)=0
\end{equation}
for all $v_1,\ldots,v_4\in \osc_pf$ and each $x\in T_pM$\,.
\item The following two equations hold for all $p\in M$ and $x,y,z,z'\in T_pM$\,,
\begin{align}\label{eq:semiparallel_1}
&R^\bot(x,y)\,h(z,z')=h(R^M(x,y)\,z,z')+h(z,R^M(x,y)\,z')\;,\\
\label{eq:semiparallel_2}&\fetth(R^N(x,y)\,z-[\fetth(x),\fetth(y)]\,z)=[R^N(x,y)-[\fetth(x),\fetth(y)],\fetth(z)]\;.
\end{align}
\item $\osc f$	is a parallel vector subbundle of the pullback bundle $f^*TN$
  (where the latter is equipped with the connection induced	 by $\nabla^N$). Hence, we
have $R^N(x,y)(\osc_pf)\subset\osc_pf$ and $R^N(x,y)|_{\osc_pf}$ is the
corresponding curvature endomorphism of $\osc_pf$ for all $x,y\in T_pM$\,.
\end{enumerate}
\end{proposition}
\begin{proof}
Part~(a) follows from the Codazzi Equation. For~(b) cf.~\cite[Proposition~3.12]{J1}. For~(c) cf.~\cite[Corollary~3.13]{J1}. For~(d) see~\cite[Corollary~3.16]{J1}. For~\eqref{eq:semiparallel_1} see~\cite[Definition~1 and Proposition~2]{J1}.
For~\eqref{eq:semiparallel_2}, note that both sides of this equation are
elements of $\so(\osc_pf)_-$\,.
Thus by virtue of Lemma~\ref{le:splitting} it is enough to verify that
\eqref{eq:semiparallel_2} holds on $T_pM$\,. For this let $\tilde y\in T_pM$ be
given; then~\eqref{eq:Gauss_Ricci} combined with the fact that $\fetth$ maps into $\so(\osc_pf)_-$ implies that
$$
[R^N(x,y)-[\fetth(x),\fetth(y)],\fetth(z)]\,\tilde
y=R^\bot(x,y)\,\fetth(z)\,\tilde y-\fetth(z)\,R^M(x,y)\,\tilde y\;;
$$
now use Eqs.~\eqref{eq:fetth1},~\eqref{eq:semiparallel_1} and again~\eqref{eq:Gauss_Ricci}.
Part~(f) is an immediate consequence of the parallelity of $h$\,.
\qed \end{proof}

\subsection{``Reduction of the codimension''}
\label{se:2.1}
We recall the following result on the ``reduction of the codimension'' in the
sense of Erbacher~\cite{Er} (cf.\ also~\cite[Theorem~2.4]{J1}):

\begin{theorem}[Dombrowski]\label{th:dombi}
Let $N$ be a symmetric space.
If $f:M\to N$ is a parallel isometric immersion and if for some point $p\in M$ the
second osculating space $\osc_pf=T_pM\oplus \bot^1_pf$
is contained in some curvature invariant subspace $V$ of $T_oN$ ($o:=f(p)$)\,, then
$f(M)\subset\bar N$\,, where $\bar N$ denotes the totally geodesic submanifold $\exp^N_o(V)$\;.
\end{theorem}

Let $\fraki(N)$ be the Lie algebra of $\Iso(N)$ and $\fraki(N)=\frakk\oplus\frakp$ be the Cartan decomposition with respect to the base point $p$\,.
For each $X\in\fraki(N)$ we have the one-parameter subgroup
$\psi_t^X:=\exp(t\,X)$ of isometries on $N$ and the induced Killing vector
field $X^*$ on $N$ (see~\eqref{eq:fundamental_vector_field}). Furthermore,
let $\pi_1:\fraki(N)\to \so(T_pN), X\mapsto\nabla^NX^*(p)$ and $\pi_2:\fraki(N)\to T_pN, X\mapsto X^*(p)$
be the canonical projections, see Appendix~\ref{se:B}.

\begin{lemma}\label{le:unique}
Let $N$ be a symmetric space.
\begin{enumerate}
\item Suppose that $V$ is a linear subspace of $T_pN$ which is not contained in any
proper curvature invariant subspace of $T_pN$\,. If $T_pg|_V=\Id$ holds for some $g\in K$\,, then
$g=\Id$\,. Therefore, if $\pi_1(X)|_V=0$ and $\pi_2(X)=0$\,, then $X=0$\,.
\item
Let $G\subset\Iso(N)$ be a connected Lie subgroup, $\frakg$ denote the
Lie algebra of $G$ and let $M$ denote the orbit $=G\,p$\,.
Then we have:
\begin{align}\label{eq:Gauss_2}
&\forall g\in G:g(T_pM)=T_{g(p)}M\;.\\
&\forall x,y\in T_pM:T_pg\,h(x,y)=h(T_pg\,x,T_p\,y)\;.
\label{eq:Gauss_3}
\end{align}
\item
In the situation of Part~(b), suppose additionally that $M$ is a full parallel submanifold.\footnote{The result remains true even if $M$ is not parallel. But then one has to use a more general version of
Theorem~\ref{th:dombi}.} Then $G$ acts effectively on $M$\,.
In particular, if $X^*|_M=0$ holds for some $X\in\frakg$\,, then $X=0$\,.
\end{enumerate}
\end{lemma}
\begin{proof}
For~(a): If $T_pg|_V=\Id$ holds for some $g\in K$\,, then $V:=\Menge{v\in
T_pN}{T_pg\,v=v}$ is a curvature invariant subspace with $\osc_pM\subset V$\,.
Thus $T_pg=\Id$ besides $g(p)=p$\,, and therefore we have $g=\Id$ since the
isotropy representation of $K$ is faithful. If $X\in\frakk$ satisfies
$\pi_1(X)|_V=0$ and $\pi_2(X)=0$\,, then $X\in\frakk$
(see~\eqref{eq:frakk_1}) and $\rho_*(X)|_V=0$ (see Proposition~\ref{p:Kovv_X}), hence $T_p\psi_t^X$ is the identity
on $V$ for each $t\in\R$\,. Thus $\psi_t^X=\Id$ by the previous, i.e. $X=0$\,.

For~(b): Since $M$ is a $G$-orbit,~\eqref{eq:Gauss_2} is straightforward and~\eqref{eq:Gauss_3} is a consequence of the Gau{\ss} equation.

For~(c): In this situation, the second osculating space	 $V:=\osc_pM$
satisfies the hypothesis for Part~(a) (because of
Theorem~\ref{th:dombi}). Now assume that $g\in G$ satisfies $g|_M=\Id$\,; then $T_pg|_{T_pM}=\Id$ by means
of~\eqref{eq:Gauss_2}\,. Thus, using also~\eqref{eq:Gauss_3}, even $T_pg|_{\osc_pM}=\Id$ and therefore
$g=\Id$ as a consequence of Part~(a). Thus $G$ acts effectively on $M$\,.
Furthermore, if $X\in\frakg$ satisfies $X^*|_M=0$\,, then $\psi_t^X$ is the identity on
$M$ for each $t\in\R$\,. Hence the previous implies that
$\psi_t^X=\Id$ for each $t\in\R$\,, i.e. $X=0$\,.
\qed \end{proof}

\section{Infinitesimal models}\label{se:3}
A well known result from~\cite{St} states that a complete parallel
submanifold $M\subset N$ is uniquely determined by its 2-jet $(T_pM,h_p)$
at one point $p\in M$\,. Conversely, let a point $p\in N$\,,
a linear subspace $W\subset T_pN$ and a symmetric bilinear
map $b:W\times W\to W^\bot$ be given (in the following called a (formal) {\em 2-jet}\/ at $p$).
The following question is somehow more delicate: Does there
{\em exist}\/ a parallel submanifold $M\subset N$ with extrinsically
homogeneous	 holonomy bundle which passes through $p$ and whose 2-jet at $p$ is
given by $(W,b)$\,?\footnote{In~\cite{JR}, the analogous problem was solved for
 arbitrary parallel submanifolds.} If the ambient space is a symmetric space,
then the answer will be given by Theorem~\ref{th:2-jet} below.

Let $\Iso(N)^0$ denote the connected component of $\Iso(N)$\,.
The {\em isotropy subgroup}\/ of $\Iso^0(N)$ and the {\em isotropy
representation}\/ at $p$ are given by
\begin{align}\label{eq:K}
&K:=\Menge{g\in\Iso^0(N)}{g(p)=p}\;,\\
&\rho:K\to\SO(T_pN),\; g\mapsto T_pg\;.
\label{eq:rho}
\end{align}
Then $\rho$ is a faithful representation
(because an isometry is determined by its value and differential at one point).
Let $\frakk$ be the Lie algebra of $K$ and $\rho_*:\frakk\to\so(T_pN)$
be the induced representation called the {\em linearized isotropy representation}, i.e.
\begin{equation}\label{eq:rho_star}
\forall X\in\frakk,\forall\,u\in T_pN:\rho_*(X)=\frac{\diff}{\diff t}\Big|_{t=0}T_p\exp(t\,X)\,u\;.
\end{equation}

Given a 2-jet $(W,b)$ at $p$\,, we introduce its ``first normal space''
$\bot^1b:=\Spann{b(x,y)}{x,y\in W}$ and the associated ``second osculating
space'' $\osc b:=W\oplus \bot^1b$\,. Then the splitting $\so(\osc b)=\so(\osc
b)_+\oplus\so(\osc b)_-$ and the linear map $\fettb:W\to\so(\osc b)_-$
are defined similar as in Section~\ref{se:2}.

\begin{definition}\label{de:model}
Let a symmetric space $N$\,, some $p\in M$ and a 2-jet $(W,b)$ at $p$ be given. Motivated by~\cite{Co},
we will call $(W,b)$ an {\em infinitesimal model}\/ if the following properties hold:
\begin{itemize}
\item $W$ is a curvature invariant subspace of $T_pN$\,, i.e. $R^N(W\times
  W)\,W\subset W$\,.
\item
Equation~\ref{eq:semiparallel_2} holds, i.e.
\begin{equation}
\fettb\big(R^N(x,y)\,z-[\fettb(x),
\fettb(y)]\,z\big)\,v=[R^N(x,y)-[\fettb(x),\fettb(y)],\fettb(z)]\,v\
\label{eq:fettb_is_semiparallel}
\end{equation}
for all $x,y,z\in W$\,, $v\in T_pN$\,.
\item For each $x\in W$ there exists some $X\in\frakk$ such that $A:=\rho_*(X)$ satisfies
\begin{equation}\label{eq:Gamma}
A(\osc b)\subset \osc b\qmq{and}A|_{\osc b}=\fettb(x)\;.
\end{equation}
\end{itemize}
\end{definition}
Note that $\rho_*(\frakk)=\so(T_pN)$ if and only if $N$ is a space form.

\begin{theorem}\label{th:2-jet}
In the situation of Definition~\ref{de:model}, the 2-jet $(W,b)$ is an
infinitesimal model if and only if there exists a parallel submanifold with
extrinsically homogeneous holonomy bundle which passes through $p$ and whose 2-jet at
$p$ is given by $(W,b)$\,.
\end{theorem}

\subsection{Proof for the ``only if direction'' of Theorem~\ref{th:2-jet}}
\label{se:3.1}
Let $M\subset N$ be a parallel submanifold with extrinsically homogeneous
holonomy bundle which passes through $p$ such that $(W,b)=(T_pM,h_p)$\,.
Then $W$ is a curvature invariant subspace of $T_pN$\,, according to
Proposition~\ref{p:properties}, and~\eqref{eq:fettb_is_semiparallel} holds by means
of~\eqref{eq:semiparallel_2}. It remains to establish the existence of a
solution to~\eqref{eq:Gamma}. To this end, let $G\subset\Iso(N)$ be a connected Lie subgroup
which has the properties described in
Definition~\ref{de:homogeneous_holonomy} and let $\frakg$ denote its Lie algebra.
Again, let $\pi_1:\fraki(N)\to \so(T_pN)$ and $\pi_2:\fraki(N)\to T_pN$
denote the canonical projections (see Section~\ref{se:2.1}).

We set
\begin{equation}\label{eq:frakm_0}
\frakm_0:=\Menge{X\in\frakg}{x:=\pi_2(X)\in T_pM\ \text{and}\ A:=\pi_1(X)\
\text{solves}~\eqref{eq:Gamma}}\;.
\end{equation}
Let $x\in T_pM$	 be given. We claim that there exists some $X\in\frakm_0$ with $\pi_2(X)=x$\,:

First, we assume additionally that $M$ is full in $N$\,. 
Then Lemma~\ref{le:unique}~(c) in combination with
Definition~\ref{de:homogeneous_holonomy} shows
that we can apply Proposition~\ref{p:KN}; thus there exists a reductive
complement $\frakm\subset\frakg$ such that $\nabla^M$ is the
corresponding canonical connection on $TM$\,.
In particular, there exists some $X\in\frakm$ with $x=\pi_2(X)$\,. In order
to show that $X\in\frakm_0$ holds\,, 
let $\gamma$ be the geodesic of $M$ with $\dot\gamma(0)=x$\,.
Following Example~\ref{ex:geodesics}, we have
\begin{equation}\label{eq:pardisp_3}
\gamma(t)=\exp(t\,X)(p)\qmq{and}\forall y\in T_pM:\ghdisp{0}{t}{\gamma}{M}\,y=T_p\exp(t\,X)\,y\;.
\end{equation}
From~\eqref{eq:pardisp_3} and~\eqref{eq:Kovv_X} it follows that
\begin{equation}
\forall y\in T_pM:\nabla^M_y(X|_M)=0\;.
\end{equation}
Thus, on the one hand, the Gau{\ss} equation yields that
\begin{equation}\label{eq:piN_2(X)}
\pi_1(X)|_{T_pM}=h(x,\cdot)\in\Hom(T_pM,\bot^1_pM)\;.
\end{equation}
On the other hand, if $x_1,\ldots,x_k$ is a basis of $T_pM$\,, then, according to~\eqref{eq:pardisp_3},
the sections $T_p\exp(t\,X)\,x_i$ define a parallel frame of $TM$ along $\gamma$\,; hence by
$\xi_{i,j}(t):=h(x_i(t),x_j(t))$ are also defined parallel sections of
$\bot M$ along $\gamma$ (since $M$ is a parallel submanifold).
Using Lemma~\ref{le:unique}~(b) (with $g:=\psi_t:=\exp(t\,X)$)\,,
we obtain that
$$
T_p\psi_t\,h(x_i,x_j)=h(T_p\psi_t\,x_i,T_p\psi_t\,x_j)
=\xi_{i,j}(t)\;,
$$
thus \begin{equation}\label{eq:piN_2(X)_again}
\pi_1(X)\,\xi_{i,j}\stackrel{\eqref{eq:Kovv_X}}=\frac{\nabla^N}{d
  t}\Big|_{t=0}\xi_{i,j}(t)=-S_{\xi_{i,j}}\,x\in T_pM\;.
\end{equation}
Since the vectors $h(x_i,x_j)$ span $\bot^1_pM$\,, \eqref{eq:fetth1},~\eqref{eq:fetth2},~\eqref{eq:piN_2(X)},~\eqref{eq:piN_2(X)_again}
imply that $X\in\frakm_0$ holds. The result follows.

Second, we consider the general case. Let $\bar M$ be a
totally geodesic submanifold of $N$ which contains $M$ such that $\bar M$
is minimal with this property (among all totally geodesic submanifolds of
$N$). Clearly, then $M$ is a full parallel submanifold of $\bar M$ whose
2-jet at $p$ is given by $(W,b)$\,, too; in particular, we have $\osc b\subset
T_p\bar M$\,. Furthermore, $\bar M$ is uniquely determined; in fact, $\bar M$ is
necessarily the connected component of the intersection of all the totally
geodesic submanifolds of $N$ which contain $M$\,.
Therefore, $G$ leaves also $\bar M$ invariant and thus $M$ has extrinsically homogeneous
holonomy bundle also in $\bar M$\,.
Even more, the Lie subgroup $\bar G\subset\Iso(\bar M)$
which is obtained as the image of the restriction map $G\to \Iso(\bar M)\,, g\mapsto
g|_{\bar M}$ has the properties described in Def.~\ref{de:homogeneous_holonomy} (now for ambient space $\bar M$).
This reduces our claim to the first case.

Finally, we decompose $X$ as $Y+Z\in\frakk\oplus\frakp$ and we recall that
$\pi_1(X)=\rho_*(Y)$ holds (see~\eqref{eq:frakp_1} and
Proposition~\ref{p:Kovv_X}).
Thus~\eqref{eq:frakm_0} implies that $A:=\rho_*(Y)$ solves~\eqref{eq:Gamma}\,.
\qed

\subsection{An extrinsic analogue of the Nomizu construction for parallel
  submanifolds with extrinsically homogeneous holonomy bundle}
\label{se:3.2}
At the end of this section, we will give the proof for 
the ``if-direction'' of Theorem~\ref{th:2-jet}.
Let $N$ be a symmetric space and let an infinitesimal model $(W,b)$ at $p$ be given.
Using similar ideas as presented in~\cite[p.~318]{Co}, we will now associate a
Lie subalgebra $\frakg\subset\fraki(N)$ with the infinitesimal
model. This can be seen as an ``extrinsic analogue of the Nomizu construction'' (cited
from~\cite[p.~315]{Co}) - but now for for parallel submanifolds 
with extrinsically homogeneous holonomy bundle in arbitrary ambient symmetric spaces.

Motivated by~\eqref{eq:frakm_0}, we define a subspace of $\fraki(N)$ via
\begin{equation}\label{eq:frakm}
\frakm:=\Menge{X\in\fraki(N)}{x:=\pi_2(X)\in W\ \text{and}\ A:=\pi_1(X)\
\text{solves}~\eqref{eq:Gamma}}\;.
\end{equation}

\begin{lemma}\label{le:triple}
Let $N$ be a symmetric space, $(W,b)$ be an infinitesimal model at $p$ and $\frakm$ be the linear space defined by~\eqref{eq:frakm}. We have
\begin{align}
\label{eq:triple1}
&[\frakm,\frakm]\subset\frakk\;,\\
&[[\frakm,\frakm],\frakm]\subset\frakm\,.
\label{eq:triple2}
\end{align}
\end{lemma}
\begin{proof}
We equip $T_pN\oplus\rho_*(\frakk)$	 with a Lie bracket according
to~\eqref{eq:bracket_1}-\eqref{eq:bracket_3}. Then by $\fraki(N)\to
\rho_*(\frakk)\oplus T_pN\,,\; X\mapsto (\pi_1(X),\pi_2(X))$
is given an isomorphism of $\Z_2$-graded Lie algebras, cf. Lemma~\ref{le:bracket}. In the
following, we replace $\fraki(N)$ by this ``alternate model'';
then $\frakk\oplus\frakp=\rho_*(\frakk)\oplus T_pN$ and $\pi_1$\,, $\pi_2$
are simply the canonical projections onto the first and second factor\,, respectively.

For~\eqref{eq:triple1}: Let $(A,x)$ and $(B,y)$ be arbitrary elements of $\frakm$\,.
Then the symmetry of $b$ implies that
\begin{align*}
[x,B]\stackrel{\eqref{eq:bracket_2}}{=}-B\,x
\stackrel{\eqref{eq:frakm}}=-\fettb(y)\,x=-b(y,x) =-b(x,y)=-[A,y]\;;
\end{align*}
hence, according to~\eqref{eq:bracket_1} and~\eqref{eq:bracket_3},~\eqref{eq:triple1} holds and,
moreover, \begin{align}\label{eq:triple3}
&[(A,x),(B,y)](\osc b)\subset\osc b\;, \\
&[(A,x),(B,y)]|_{\osc b}=-R^N(x,y)+[\fettb(x),\fettb(y)]\;.
\label{eq:triple4}
\end{align}
Then~\eqref{eq:triple2} follows
from~\eqref{eq:fettb_is_semiparallel},~\eqref{eq:frakm},~\eqref{eq:triple3} and~\eqref{eq:triple4}.
\qed \end{proof}

Using the previous lemma and the Jacobi identity,
we obtain that the linear space $\frakg:=[\frakm,\frakm]+\frakm$
is even a Lie subalgebra of $\fraki(N)$\,. Let $G$ denote the connected Lie
subgroup of $\Iso(N)$ whose Lie algebra is $\frakg$\,.
Furthermore, consider the orbit $M:=G\,p$\,.

\begin{lemma}\label{le:orbit}
In the situation of Lemma~\ref{le:triple}, the 2-jet  of $M$ at $p$ is given  by $(W,b)$\,.
\end{lemma}

\begin{proof}
We note that $$T_pM=\pi_2(\frakg)\stackrel{\eqref{eq:frakm}}{\subset} W$$
where the last inclusion is actually an equality since $(W,b)$ is an infinitesimal model.
To see that $b=h_p$ holds, we observe that every $X\in\frakm$ induces a
Killing vector field on $N$ which is tangent to $M$\,.
The Gau{\ss} equation yields
$$
\forall y\in
T_pM:h(y,\pi_2(X))=(\pi_1(X)\,y)^\bot
\stackrel{\eqref{eq:fetth1},\eqref{eq:frakm}}{=}b(\pi_2(X),y)=b(y,\pi_2(X))\;.
$$
The result follows, since $\pi_2|_{\frakm}$ is onto $W$\,.
\qed \end{proof}

\paragraph{Proof for the ``if-direction'' of Theorem~\ref{th:2-jet}}
Let an infinitesimal model $(W,b)$ be given. First, we assume that $\osc b$ is not contained in any proper curvature
invariant subspace of $T_p N$\,.

Hence we may define a subspace $\frakm\subset\fraki(N)$ according to~\eqref{eq:frakm}. Then
$\frakg:=[\frakm,\frakm]\oplus\frakm$ is a subalgebra of
$\fraki(N)$ and the orbit $M:=G\,p$ of the
corresponding connected Lie subgroup of $\Iso(N)$ is a parallel submanifold
of $N$ whose  2-jet at $p$ is given by $(W,b)$	(see Lemma~\ref{le:triple} and~\ref{le:orbit}).
Let us see that $M$ is a parallel submanifold with extrinsically
homogeneous	 holonomy bundle - even more: The group $G$ itself has the
properties described in Def.~\ref{de:homogeneous_holonomy}:

First, we also assume that $\osc b$ is not contained in any proper
curvature invariant subspace of $T_pN$\,. Under this assumption, we claim that $\frakm$ is a
reductive complement in $\frakg$ (see Appendix~\ref{se:A}):

>From Lemma~\ref{le:unique}~(a) we see that $\pi_2|_{\frakm}$ is injective
(here we use that $\osc b$ is not contained in any proper
curvature invariant subspace of $T_pN$) and thus $\pi_2|_{\frakm}$ is actually an
isomorphism onto $T_pM$\,. Let $H$ denote the isotropy subgroup in $G$ at $p$ and $\rho:K\to\SO(T_pN)$ denote the
isotropy representation.
Then we have the well known relation
\begin{equation}\label{eq:pi_1-pi_2}
\pi_2(\Ad(g)\,X)=T_pg\,\pi_2(X)\qmq{and}\pi_1(\Ad(g)\,X)=\Ad(T_pg)\,\pi_1(X)
\end{equation}
for each $X\in\fraki(N)$ and $g\in \rmK$\,.
Since the 2-jet of $M$ at $p$ is given by $(W,b)$\,,
we now conclude from~\eqref{eq:Gauss_2},~\eqref{eq:Gauss_3},~\eqref{eq:frakm} and~\eqref{eq:pi_1-pi_2} that $\Ad(g)\,\frakm\subset
\frakm$ holds for each $g\in H$\,. The claim follows.

As explained in Appendix~\ref{se:A}, $TN|_M$ is a homogenous vector bundle over
$M$ and the reductive complement $\frakm$ induces a canonical
connection $\nabla^c$ on $TN|_M$\,. Since the second fundamental form of $M$ is a $\nabla^c$ -parallel
section of $\Sym^2(TM,\bot M)$ according to Part~(b) of Lemma~\ref{le:parallel}, the
parallelity of $M$	will be established by showing that \begin{equation}\label{eq:nabla_can}
\nabla^c\ \text{coincides with}\ \nabla^M\oplus\nabla^\bot\ \text{on}\ \osc M\;.
\end{equation}
Furthermore, as a consequence of~\eqref{eq:pardisp},~\eqref{eq:nabla_can}
also implies that $G$ satisfies Def.~\ref{de:homogeneous_holonomy}.
Moreover, $\Delta:=\nabla^M\oplus\nabla^\bot-\nabla^c$ is a
$\nabla^c$-parallel section of $\Hom(TM,\End(TN|_M))$ and $\osc M$ is a $\nabla^c$-parallel vector
subbundle of $TN|_M$\,, too, in accordance with Lemma~\ref{le:parallel};
thus for the compliance of~\eqref{eq:nabla_can}
it suffices to show that $\Delta(x)\,v=0$ for each $x\in T_pM$ and
$v\in\osc_pM$\,:

To this end, let $\hat\Gamma:T_pM\to\frakm$ be the inverse of $\pi_2|_{\frakm}$ and consider the curve
$\gamma:\R\to M,\gamma(t):=\exp(t\,\hat\Gamma(x))(p)$\,; note that
$\dot\gamma(0)=x$ according to~\eqref{eq:frakm}.
Furthermore, by $y(t):=T_p\exp(t\,\hat\Gamma(x))\,y$ and
$\xi(t):=T_p\exp(t\,\hat\Gamma(x))\,\xi$ there are defined $\nabla^c$-parallel
sections of $TM$ and $\bot^1M$ along $\gamma$\,, respectively, for all $y\in T_pM$ and
$\xi\in\bot^1_pM$ (see~\eqref{eq:frakm333}). Thus we have
\begin{align*}
\Delta(x)\,y&=\big(\frac{\nabla^\top}{\partial
  t}-\frac{\nabla^c}{\partial t}\big)\Big
|_{t=0}y(t)=\frac{\nabla^\top}{\partial t}\Big|_{t=0}y(t)
=(\frac{\nabla^N}{\partial t}\Big |_{t=0}y(t))^\top\\
&\stackrel{\eqref{eq:Kovv_X}}{=}(\pi_1(\hat\Gamma(x))\,y)\top
\stackrel{\eqref{eq:frakm}}{=}(\fettb(x)\,y)^\top\stackrel{\eqref{eq:odd}}{=}0\;,
\end{align*}
and for similar reasons $\Delta(x)\,\xi=0$ (see~\eqref{eq:fetth2})\;.
Thus~\eqref{eq:nabla_can} is established.

Second, we consider the general case. Let $V$ be
the intersection of all the curvature
invariant subspaces of $T_pN$ which contain $\osc b$ \,. Then $V$ is curvature
invariant; hence $\bar M:=\exp(V)$ is a totally geodesic submanifold of $N$\,. Then $(W,b)$ is an
infinitesimal model for $\bar M$\,, too; moreover, $\osc b$ is not contained
in any proper curvature invariant subspace of $T_p\bar M$\,.
Now we claim that $g(\bar M)=\bar M$ holds for every $g\in G$\,:

Equation~\ref{eq:frakm} shows that $\pi_2(X)\in V$ and
$\pi_1(X)(V)\subset V$ for every $X\in\frakm$\,; therefore, and since $\bar M$ is totally
geodesic in $N$\,,	we even have
$X^*(\bar M)\subset T\bar M\qmq{and}X^*|_{\bar M}\in\bar\frakm$\,.
Since $G$ is connected and $\frakg=[\frakm,\frakm]\oplus\frakm$ holds, our claim now follows.

Thus $G\to\Iso(\bar M),g\mapsto g|_{\bar M}$ defines a Lie group homomorphism
onto a connected Lie subgroup $\bar G\subset\Iso(\bar M)$
such that the induced Lie algebra homomorphism maps $\frakm$ onto the subspace $\bar
\frakm\subset\fraki(\bar M)$ which is defined on the analogy
of~\eqref{eq:frakm}. Then $[\bar\frakm,\bar\frakm]\oplus\bar\frakm$ is the Lie
algebra of $\bar G$ and $M=\bar G\,p$\,.
Applying the first case, we now see that $M$ is a parallel submanifold of
$\bar M$ and $\bar G$ has the properties described in
Def.~\ref{de:homogeneous_holonomy} (for ambient space $\bar M$).
Thus $M$ is a parallel submanifold of $N$\,, too, (since $\bar M$ is totally
geodesic in $N$) and $G$ has the properties described in
Def.~\ref{de:homogeneous_holonomy}. This finishes the proof.
\qed

\subsection{On the existence of non-symmetric full parallel submanifolds}
\label{se:3.3}
\begin{theorem}\label{th:new_examples?}
Let symmetric spaces $N$ and $\tilde N$ with base points $p\in M$ and $o\in \tilde M$ be
given. Let $\fraki(N)=\frakk\oplus\frakp$ and $\rho_*:\frakk\to\so(T_oN)$ denote the corresponding Cartan
decomposition and the linearized isotropy representation of $N$\,, respectively (similarly for $\tilde N$).
Suppose that there exists a full symmetric submanifold $\tilde M\subset\tilde N$ through $o$
and some proper linear subspace $V\subset T_pN$ which is not contained in any proper curvature invariant subspace of $T_pN$\,,
a linear isometry $F:T_o\tilde N\to V$ and a Lie algebra homomorphism $\hat F:\tilde\frakk\to
\frakk$ such that
\begin{align}\label{eq:cond_1}
&\forall x,y\in T_o\tilde M, v\in T_p\tilde N:F(R^{\tilde N}(x,y)\,v)=R^N(F\,x,F\,y)(F\,v)\;,\\
\label{eq:cond_2}
&\forall X\in\tilde\frakk,v\in T_o\tilde N: F(\tilde\rho_*(X)\,v)=\rho_*(\hat F(X))(F\,v)\;.
\end{align}
Then there exists a full parallel submanifold $M\subset N$ with extrinsically homogeneous  holonomy bundle
which is not extrinsically symmetric in $N$\,. More precisely, the 2-jet of
$M$ at $p$ is given by $(W,b)$\,, where $W$ denotes the linear space $F(T_o\tilde M)$ and $b:W\times W\to W^\bot$ is the bilinear map
characterized by $b(F\,x,F\,y)=F(h^{\tilde M}(x,y))$ for all $x,y\in T_o\tilde M$\,.
\end{theorem}
One can show that every full extrinsically homogeneous circle is induced by a symmetric submanifold $\rmS^1(r)\subset \R^2$ ($r>0$) as described in Theorem~\ref{th:new_examples?}; but it is not clear whether there are more examples.

\begin{proof}
We claim that $(W,b)$ is an infinitesimal model in $N$ in the sense of
Definition~\ref{de:model}:

Because $T_o\tilde M$ is curvature invariant,~\eqref{eq:cond_1} implies that $W$
is curvature invariant, too; moreover the symmetry of $h^{\tilde M}_o$ implies
by means of~\eqref{eq:cond_1} that $b$ is symmetric, and for the same
reason~\eqref{eq:fettb_is_semiparallel} holds, too.
Since $\tilde M$ is extrinsically symmetric in $\tilde N$\,, for every $\tilde x\in T_o\tilde M$ there
exists some $\tilde X\in\tilde \frakk$ with $\tilde\rho_*(\tilde
X)=\tilde\fetth(\tilde x)$ (cf.~\cite[Theorem~4]{E2})\,; then by means of~\eqref{eq:cond_2},
$x:=F(\tilde x)$ together
with $A:=\rho_*(\hat F(\tilde X))$ gives a solution to~\eqref{eq:Gamma}.
Thus Theorem~\ref{th:2-jet} exhibits the existence
of a parallel submanifold $M\subset N$ with extrinsically homogeneous
holonomy bundle such that $p\in M$\,, $T_pM=W$ and $h_p=b$\,.
Then $M$ is full in $N$ because of Theorem~\ref{th:dombi} (since $V$ is not contained in any proper curvature invariant subspace of $T_pN$).
On the other hand, every full symmetric submanifold of $N$ is even 1-full according to~\cite[Theorem~3.2]{J1}, hence $M$ is not extrinsically symmetric in $N$ (since $V$ is strictly contained in $T_pN$).
\qed \end{proof}

\subsection{A characterization of extrinsic homogeneity}
\label{se:3.4}
The following theorem seems to be new, although weaker
results are well known\footnote{For a weaker version of
  ``(a)\;$\Rightarrow$\;(b)'' see~\cite[Theorem~3]{E2}\,. For a symmetric
submanifold ``(b)\;$\Rightarrow$\;(a)'' follows from~\cite[Theorem~4]{E2}\,. For a circle (i.e. $M\cong \R$),
Theorem~\ref{th:homogen} can easily be derived from~\cite[Corollary~1.4]{MT}
combined with Example~\ref{ex:homogeneous_holonomy} of this article.}:

\begin{theorem}\label{th:homogen}
Let a parallel isometric immersion $f$ from a simply connected symmetric space $M$ into the symmetric space $N$ be given.
Let $p\in M$ be some point and $(W,b):=(T_pM,h_p)$ be the 2-jet of $M$ at $p$\,.
Then the following two assertions are equivalent:
\begin{enumerate}
\item
$f(M)$ is a submanifold with extrinsically homogeneous holonomy bundle.
\item
For each $x\in T_pM$ there exists some $X\in\frakk$ such that $A:=\rho_*(X)$ satisfies~\eqref{eq:Gamma}.
\end{enumerate}
If any of these assertions holds, then $f(M)$ is a parallel submanifold of $N$
and $f:M\to f(M)$ is a Riemannian covering.
\end{theorem}

\begin{proof}
``$(a)\Rightarrow (b)$'': If~(a) holds, then $f:M\to f(M)$ is necessarily a
Riemannian covering:

In fact, the induced map $f:M\to f(M)$ is differentiable (since $f(M)$ is
quasi-regular, hence this map
is necessarily a local isometry (by Sard's theorem), i.e. it is a Riemannian
covering (because $M$ is complete). In particular, then $f(M)$ is a parallel submanifold of $N$\,; hence our result follows from Theorem~\ref{th:2-jet}\,.

For ``$(b)\Rightarrow (a)$'': Consider the 2-jet of $f$ at $p$\,, given by $(W,b):=(T_pM,h_p)$\,.
In order to apply Theorem~\ref{th:2-jet}, we make the following observations:
\begin{itemize}
\item $T_pM$ is a curvature invariant subspace of $T_{f(p)}N$ according to Proposition~\ref{p:properties}\,.
\item~\eqref{eq:fettb_is_semiparallel} holds because
of~\eqref{eq:semiparallel_2}.
\item By assumption, for each $x\in T_pM$ there exists $X\in\frakk$ such
  that~\eqref{eq:Gamma} holds.
\end{itemize}
Thus $(W,b)$ is an infinitesimal model, hence Theorem~\ref{th:2-jet} exhibits the existence of a
parallel submanifold $\tilde M$ with extrinsically homogeneous	holonomy bundle through $f(p)$
whose 2-jet at $f(p)$ is given by $(W,b)$\,. 
We claim that $f(M)=\tilde M$ holds:

First note that both $M$ and $\tilde M$ are complete Riemannian manifolds.
Let $\gamma:\R\to M$ be a geodesic line of $M$ through $p$\,, parameterized by arc
length, and consider the curve $c:=f\circ\gamma:\R\to N$\,. According to~\cite{St} or~\cite[p.~8-14]{JR},
there exists an orthonormal frame $e_1,\ldots e_{r}$ along $c$
and positive numbers $\kappa_1,\ldots,\kappa_{r-1}$ such that the
``Frenet equations'' $$\nabla^N e_i=-\kappa_{i-1}\,e_{i-1}+\kappa_i\,e_{i+1}$$
with $e_1=\dot c$ hold for $i=1,\ldots,r$ (here we set $\kappa_0:=\kappa_{r}:=0$ and $e_0:=e_{r+1}:=0$).
Furthermore, the coefficients $\kappa_i$ and the orthonormal frame $e_1,\ldots,
e_{r}$ are uniquely determined by $b$ and the velocity vector $\dot c(0)$\,.
Ditto for the geodesic lines of $\tilde M$ through $f(p)$\,. Because
$f$ and $\tilde M$ have the same 2-jet at $p$ and $f(p)$\,, respectively,
we now see that $f\circ \exp^M|_{T_pM}=\iota^{\tilde M}\circ \exp^{\tilde
  M}|_{T_{f(p)}\tilde M}$ holds\,. The result follows.
\qed
\end{proof}


\subsection{Proof of Example~\ref{ex:extrinsic_holonomy}}
Set $N:=\C\rmP^1\times\C\rmP^1$ and let $p:=((1:0),(1:0))$\,; then $T_pN\cong
\C^2$\,. Fix mutually orthogonal unit vectors $x=(x_1,x_2)$ and $y=(y_1,y_2)$ in $\C^2$\,.
Let $\kappa$ be a positive constant and let $c:\R\to N$
be the unique solution to the ordinary differential equation
$$\Kodel{N}{\Kodel{N}{\dot c}}(t)=-\kappa^2\dot c(t)
\qmq{with}c(0)=p,\ \dot c(0)=x\qmq{and}
\Kodel{N}{\dot c}(0)=\kappa\,y\;.$$
Then $c$ is a circle of curvature $\kappa$
(see~\cite[p.~1]{MT}) and the linear span $V:=\{x,y\}_\R$ is the second osculating space of $c$ at $0$\,.

Let $J$ denote the standard complex structure of $\C^2$ and set $\i:=\sqrt{-1}$\,;
then $\tilde J:=\diag(\i,-\i)$ (diagonal matrix) defines a second complex structure on $\C^2$\,.
Suppose that $V$ has the following properties:
\begin{itemize}
\item  The canonical projections $\C^2\to \C$ onto the first and second factor,
  respectively, both induce isomorphisms $V\to\C$\,.
\item $V$ is neither a complex nor a totally real subspace of $\C^2$ (with
  respect to $J$)\,.
\item $V$ is not $\tilde J$-invariant.
\end{itemize}
Note that the subset of the Grassmannian of 2-planes in $\C^2$ with these
properties is open and dense; thus our conditions are ``generic''.

We claim that $c$ is full:
By contradiction, suppose that there exists a proper totally
geodesic submanifold $M$ with $c(\R)\subset M$\,;
thus $T_pM$ is a curvature invariant subspace of $T_pN$ whose
dimension is two or three. W.l.o.g., we may assume that $M$ is maximal in $N$ with this property, and hence
$T_pM$ is a maximal proper curvature invariant subspace of $T_pN$ with
$V\subset T_pM$\,.
It is well known that $N$ is holomorphically isometric to the complex hypersurface
$\scrQ^2(\C):=\Menge{[z_0:\cdots:z_3]\in\C\rmP^3}{z_0^2+\cdots z_3^2=0}$ of
$\C\rmP^3$\,, usually called the (2-dimensional) ``complex
quadric''.\footnote{In fact, the ``Segre  embedding''
  $\C\rmP^1\times\C\rmP^1\to \C\rmP^3$ is onto $\scrQ^2(\C)$\,.}
Using the classification of totally geodesic submanifolds of the complex
quadric (see~\cite[Theorem~4.1 and Sec.~5]{Kl}), we thus infer that $M$ is necessarily of the following ``type'':
\begin{itemize}
\item ``Type $(G3)$'':\ Then
there exists a totally geodesic embedding $\rmS^1\hookrightarrow \C\rmP^1$ and either $M\cong\C\rmP^1\times\rmS^1$
or $M\cong\rmS^1\times \C\rmP^1$ such that the product structure of $M$ is compatible with the product structure of
$N$\,.
\item ``Type $(P1,2)$'':\ Then $\dim (M)=2$ and $M$ is a totally real submanifold\,.
\item ``Type $(P2)$'': Then $\dim(M)=2$ and $M$ is a complex submanifold\,.
\end{itemize}

Now we see: If $M$ is of Type (G3), then one of the canonical projections from
$V$ to $\C$ is not surjective, which is contrary to our assumptions.
If $M$ is of Type~$(P1,2)$ or $(P2)$\,, then $V$ is a complex or a totally real subspace of
$\C^2$\,, a contradiction again. Thus $c$ is full.

We also claim that Assertion~(b) of Theorem~\ref{th:homogen}
is not valid here, i.e. $c(\R)$ is not a homogeneous submanifold of $N$:

Suppose, by contradiction, that there exists some
$A\in\frakk$ such that linear map $A:=\rho_*(X)$
satisfies $A(V)\subset V$ and $A|_V=\fettb(x)$\,.
Clearly, $\frakk\cong\fraku(1)\oplus\fraku(1)$ and $\rho_*$ is given canonically;
hence there exists $A_i\in\fraku(1)$ for $i=1,2$ such that $A=A_1\oplus A_2$\,.
Then there exist real numbers $\mu_i$ such that $A_i$ is given by
multiplication with $\i\,\mu_i$	 for $i=1,2$\,. Then either
$\mu_1=\pm \mu_2$ and hence $A_1=\pm A_2$ holds (but then $V$ would be
complex or $\tilde J$-invariant, which is not possible)
or the set $\{\pm \mu_1,\pm \mu_2\}$ has cardinality equal to 4. In the latter
case, no proper $A$-invariant subspaces of $\C^2$ besides its two product
factors exist (not even over the real numbers). Since $V$ is $A$-invariant,
$V$ would therefore be one of the product factors, which is contrary to our assumptions.
\qed

\section{Geometry of irreducible parallel submanifolds}
\label{se:parallel_flat}
Let $N$ be a symmetric space, $M$ be a simply connected symmetric space and $f:M\to N$ be a parallel isometric immersion.
In the following, again we implicitly identify $T_pM\cong
Tf(T_pM)$\,; for convenience, the reader may assume that $M\subset N$ is a submanifold and $f=\iota^M$\,.
According to Proposition~\ref{p:properties}, $T_pM$ is even a curvature invariant subspace of
$T_{f(p)}N$\,. Hence, by virtue of a result due to E.~Cartan, \begin{equation}
\label{eq:tg} \bar M:=\exp^N(T_pM)\subset N
\end{equation}
is a totally geodesic submanifold of $N$ (where $\exp^N:T_oN\to N$ denotes the exponential spray).

Let $p\in M$ be a ``base point'' and $\hol(M)$\,, $\hol(\bar M)$ and $\hol(N)$ denote the holonomy Lie algebras of
$M$\,, $\bar M$ and $N$\,, respectively.
Since $\bar M$ is totally geodesic, the curvature tensor of $\bar
M$ at $p$ is given by $R^N|_{T_pM\times T_pM\times T_pM}$\,. Moreover, the curvature
tensors of $M$\,, $\bar M$ and $N$ are parallel, respectively; thus the Theorem of Ambrose/Singer implies that
\begin{align}
\label{eq:hol(M)}
&\hol(M)=\Spann{R^M(x,y)}{x,y\in T_pM}\subset\so(T_pM)\;,\\
\label{eq:hol(N)}
&\hol(N)=\Spann{R^N(u,v)}{u,v\in T_{f(p)}N}\subset\so(T_{f(p)}N)\;,\\
\label{eq:hol(M_top)_2}
&\hol(\bar M)=\Spann{R^N(x,y)|_{T_pM}}{x,y\in T_pM}\subset\so(T_pM)\;.
\end{align}
The following result is a consequence of~\cite[Theorem~2.4]{Ts}:

\begin{proposition}[K. Tsukada~\cite{Ts}]\label{p:hol-relations}
Let $N$ be a symmetric space, $M$ be a simply connected symmetric space and $f:M\to N$ be a parallel isometric immersion. Suppose additionally that $M$ is irreducible and that $\dim(M)\geq 2$\,. Then there are exactly the following two possibilities: \begin{itemize}
\item Either $\hol(\bar M)$ acts irreducibly on $T_pM$ (i.e. the totally geodesic submanifold $\bar M$ defined in~\eqref{eq:tg} is a flat),
\item or $\hol(\bar M)$ acts irreducibly on $T_pM$ (i.e. $\bar M$ is a locally irreducible symmetric space\footnote{\label{Fn:irreducible}It is known that a symmetric space is locally irreducible if and only if its universal covering space is irreducible.}).
\end{itemize}
\end{proposition}
Since the rest of this article is based on this proposition, let me briefly outline its proof:

\begin{proof}
First, let $T_pM=\bigoplus_{i=0}^kW_i$ be a decomposition into linear subspaces such
that $W_0$ is the maximal subspace on which $\hol(\bar M)$ acts trivially
and $W_i$ is an irreducible $\hol(\bar M)$-module for $i\geq 1$\,.
By virtue of the ``de Rham Decomposition Theorem''
(see~\cite[p.~290]{BCO}), the linear spaces $W_i$ are uniquely determined up to a permutation of the index
set $\{1,\ldots,k\}$\,.

Second, we use Proposition~\ref{p:properties}~(a) to define
a tensor field $\tilde R$ of type $(1,3)$ on $M$\,, via
$$\forall p\in M,x,y,z\in T_pM:\ \tilde R(x,y)\,z:=R^N(x,y)\,z\;.$$
It is well known that $\tilde R$ is a parallel tensor field (see~\cite[Lemma~2.3]{Ts}).

Using the parallelity of $\tilde R$ and the uniqueness of the linear spaces $W_i$\,, now one can show that $g_t(W_i)=W_i$ holds for every curve $g_t$ into the Holonomy Lie group of $M$  and each $i$\,; i.e. $W_i$ is also $\hol(M)$-invariant. Since $T_pM$ is an irreducible $\hol(M)$-module (again because of de Rhams Theorem), we hence conclude that either $T_pM=W_0$ or $T_pM=W_1$ holds. In the first case, $\bar M$ is a flat;
in the second case, $\bar M$ is a locally irreducible Riemannian space.
Since $\dim(M)\geq 2$\,, the two possibilities also exclude each other.
\qed
\end{proof}

\subsection{The case that $\bar M$ is a flat}
\label{se:Theorem4}
In this section, we will establish the following theorem:

\begin{theorem}\label{th:4}
Let a parallel isometric immersion $f$ from a simply connected irreducible
symmetric space $M$ into a symmetric space $N$ of compact or non-compact
type be given. Suppose additionally that the totally geodesic submanifold $\bar M$ defined by~\eqref{eq:tg} is a flat. Then even $f(M)$ is contained in some flat of $N$\,.
\end{theorem}

Let $N$ be a symmetric space and let $\Gamma:T_pN\to\frakp$ be the inverse of the canonical isomorphism
$\frakp\to T_pN$\,.

\begin{proposition}\label{p:curved_flat}
Let $N$ be a symmetric space of compact or non-compact type and let a linear subspace $W\subset T_pN$ be given. The following is equivalent:
\begin{enumerate}
\item $W$ is a curvature isotropic subspace of $T_pN$\,.
\item $[\Gamma(u),\Gamma(v)]=0$ for all $u,v\in W$\,.
\item $\exp^N(W)$ is a flat.
\item The sectional curvature of $N$ vanishes on every 2-plane of $W$\,, i.e.
$\g{R^N(u,v)\,v}{u}=0$ for all $u,v\in W$\,.
\end{enumerate}
\end{proposition}

\begin{proof}
 $(a)\Leftrightarrow(b)$  is an immediate consequence of
Lemma~\ref{le:bracket}\,.
 For $(b)\Leftrightarrow(c)$
 cf.~\cite[Ch.~V,~Proposition~6.1]{He}. $(c)\Rightarrow~(d)$ is obvious.
Let me give a proof for $(d)\Rightarrow~(b)$ in case $N$ is irreducible :
Via the canonical isomorphism $\frakp\to T_pN$\,,
the metric is given at $p$ by a
multiple $c\neq 0$ of the Killing form $B$ of $\fraki(N)$ restricted to
$\frakp$\,; where without loss of generality we may assume that $c\in\{-1,1\}$\,. Let two
orthonormal vectors $u,v\in T_pN$ be given and denote by $\kappa_{u,v}$ the sectional
curvature of the 2-plane spanned by $\{u,v\}$\,. Then, according
to~\cite[Ch.~V,~\S~3, Equation~2]{He},
$$
\kappa_{u,v}=c\,B([\Gamma(u),\Gamma(v)],[\Gamma(u),\Gamma(v)])\;;
$$
hence $\kappa_{u,v}=0$ forces $[\Gamma(u),\Gamma(v)]=0$ by means of the (positive or negative) definiteness
of $B$\,. $(d)\Rightarrow~(b)$ immediately follows. For the general case,
cf.~\cite[Ch.~V,~\S~3]{He}, and use (1) there instead of~(2).
\qed \end{proof}

\begin{lemma}\label{le:parallel_flat_1}
Let $f:M\to N$ be a parallel isometric immersion and suppose that $N$ is a
symmetric space of compact or non-compact type. If the totally geodesic submanifold $\bar M$~\eqref{eq:tg} is a flat, then
$T_pM$ and $\bot^1_pf$ both are curvature isotropic subspaces of $T_pN$.
\end{lemma}
\begin{proof}
For $T_pM$ use Proposition~\ref{p:curved_flat}. For $\bot^1_pf$\,, let
$\xi,\eta\in\bot^1_pf$ be given.
According to Proposition~\ref{p:curved_flat},
it suffices to show that $\g{R^N(\xi,\eta)\,\eta}{\xi}=0$\,:

Without loss of generality, we may assume that there
exist $x,y\in T_pM$ with $\xi=h(x,x)$\,, $\eta=h(y,y)$ (since $h$ is a symmetric bilinear map).
Since $T_pM$ is curvature isotropic, and because
$\fetth(x)\,\fetth(y)\,(T_pM)\subset T_pM$ holds\,, r.h.s.\ of
\eqref{eq:fundamental} vanishes on $\osc_pf$\,; and so $R^N(\xi,\eta)$ vanishes on $\osc_pf$\,,
too\,, according to Proposition~\ref{p:properties}~(b) . In particular, $\g{R^N(\xi,\eta)\,\eta}{\xi}=0$\,.
\qed \end{proof}

For each subspace $V\subset\so(\osc_pf)$ we introduce its centralizer in $\so(\osc_pf)$\,,
\begin{equation}\label{eq:second_subspace}
\frakc(V):=\Menge{A\in\so(\osc_pf)}{\forall B\in V:A\circ B=B\circ A}\,.
\end{equation}

\begin{proposition}\label{p:parallel_flat_2}
 Let $f$ be a parallel isometric immersion from a simply connected irreducible symmetric space $M$ into the symmetric space $N$\,. If the totally geodesic submanifold $\bar M$~\eqref{eq:tg} is a flat, then $\frakc(\fetth(T_pM))\cap\so(\osc_pf)_-=\{0\}$\,.
\end{proposition}

However, the proof of Proposition~\ref{p:parallel_flat_2} has to be postponed until
the end of this section.

\begin{proposition}\label{p:parallel_flat_3}
 Let $f$ be a parallel isometric immersion from a simply connected, irreducible symmetric space $M$ into the symmetric space $N$ of compact or non-compact type. If the totally geodesic submanifold $\bar M$~\eqref{eq:tg} is a flat, then the second osculating space $\osc_pf$ is a curvature isotropic subspace of $T_{f(p)}N$\,.
\end{proposition}

\begin{proof}
By virtue of Proposition~\ref{p:curved_flat}, it is enough to show that
$\g{R^N(v_1,v_2)\,v_3}{v_4}=0$ for all $v_1,v_2,v_3,v_4\in\osc_pf$\,. Furthermore, according to
Lemma~\ref{le:parallel_flat_1}, we have $\g{R^N(x,y)\,u}{v}=0$
and $\g{R^N(\xi,\eta)\,u}{v}=0$ for all $x,y\in T_pM$\,, $\xi,\eta\in
\bot^1_pf$ and $u,v\in\osc_pf$\,; and hence it remains to prove that $\g{R^N(x,\xi)\,u}{v}=0$ for all
$x\in T_pM$\,, $\xi\in\bot^1_pf$ and $u,v\in\osc_pf$\,. To this
end, let $x\in T_pM$
and $\xi\in\bot^1_pf$ be arbitrary, but fixed, and $A\in\so(\osc_pf)$ be the linear map
characterized by
$$
\forall u,v\in\osc_pf:\g{A\,u}{v}=\g{R^N(x,\xi)\,u}{v}\;.
$$
We claim that $A$ belongs to $\frakc(\fetth(T_pM))\cap\so(\osc_pf)_-$\,:

In fact, using the symmetries of $R^N$\,, we
have $$\forall x',y'\in T_pM:
\g{A\,x'}{y'}=\g{R^N(x,\xi)\,x'}{y'}=\g{R^N(x',y')\,x}{\xi}=0\;;$$
and furthermore, using similar arguments, $\g{A\,\xi'}{\eta'}=0$ for all
$\xi',\eta'\in\bot^1_pf$\,. Hence $A\in\so(\osc_pf)_-$\,, in accordance with~\eqref{eq:even}\,. Moreover, by means of~\eqref{eq:fetth_3}, we have for all $u,v\in\osc_pf,y\in T_pM$\,:
\begin{align*}
\g{[\fetth(y),A]\,u}{v}&=-\g{R^N(x,\xi)\,u}{\fetth(y)\,v}-\g{R^N(x,\xi)\,\fetth(y)\,u}{v}\\
&\stackrel{\eqref{eq:fetth_3}}{=}\g{R^N(h(x,y),\xi)\,u}{v}-\g{R^N(x,S_\xi\,y)\,u}{v}=0
\end{align*}
(because $(h(x,y),\xi)\in\bot^1_pf\times \bot^1_pf$ and $(x,S_\xi\,y)\in T_pM\times T_pM$)\,.
Thus $A\in\frakc(\fetth(T_pM))\cap\so(\osc_pf)_-$\,, and therefore, according
to Proposition~\ref{p:parallel_flat_2}, $A=0$\,; the result follows.
\qed \end{proof}
\paragraph{Proof of Theorem~\ref{th:4}} Set $\bar N:=\exp(\osc_pf)$\,;
then $\bar N$ is a flat, as a consequence of
Proposition~\ref{p:parallel_flat_3} and Proposition~\ref{p:curved_flat}\,.
Furthermore, $f(M)$ is contained in $\bar N$\,, by virtue of Theorem~\ref{th:dombi}\,.
\qed

\paragraph{Now we come to the proof of Proposition~\ref{p:parallel_flat_2}.}

\begin{lemma}\label{le:Le1000}
In the situation of Proposition~\ref{p:parallel_flat_2},
there exists a complete full parallel submanifold $\tilde M\subset\osc_pf$ with $0\in\tilde M$\,, $T_0\tilde
M=T_pM$ and $\tilde h_0=h_p:T_pM\times T_pM\to\bot^1_pf$\,. Moreover, $\tilde M$ is a locally irreducible symmetric space.
\end{lemma}
\begin{proof}
We consider the Euclidian vector space $V:=\osc_pf$ and the 2-jet which is given by
$(W:=T_pM,b:=h_p)$\,. Since $T_pM$ is a curvature isotropic subspace of $T_{f(p)}N$\,, we have by virtue of~\eqref{eq:semiparallel_2}
\begin{equation}\label{eq:Gleichung30}
\forall x,y,z\in W, v\in V:\fetth([\fetth(x),\fetth(y)]\,z)\,v=[[\fetth(x),\fetth(y)],\fetth(z)]\,v\;.
\end{equation}
Therefore, $(W,b)$ is an infinitesimal model in $V$ (since $V$ is a Euclidian space)
and hence, according to Theorem~\ref{th:2-jet}, there exists a parallel
submanifold $\tilde M\subset V$ through $0$ whose 2-jet at $0$ is given by
$(W,b)$\,; in particular, $\tilde M$ is a symmetric space. Using the equation of Gau\ss\,, we notice that
$R^M(x,y)\,z=R^{\tilde M}(x,y)\,z$ (since $T_pM$ is curvature isotropic in $N$)\,, i.e. $\hol(\tilde M)=\hol(M)$ and hence $\tilde M$ is locally irreducible (since $M$ has this property, see Footnote~\ref{Fn:irreducible}).
\qed \end{proof}

Motivated by the last lemma, we will now study properties of the second fundamental
form of an (intrinsically) irreducible full parallel submanifold $\tilde M$
of the Euclidian space $E$\,.

Let $\tilde N$ be a simply connected irreducible symmetric space of compact type whose
isotropy subgroup at $\tilde p$ is denoted by $\tilde K$ and whose Cartan decomposition is given
by $\fraki(\tilde N)=\tilde\frakk\oplus\tilde\frakp$\,.
Let $\tilde B$ denote the Killing form of $\fraki(\tilde N)$\,; then, since
$\fraki(\tilde N)$ is a compact, semisimple Lie algebra
(cf.~\cite[Ch.~V,~\S~1]{He}), $\tilde B$ is a negative definite,
invariant form (cf.~\cite[Ch.~II,~\S~6]{He}). Thus $\tilde\frakp$ is a
Euclidian vector space by means of the positive definite symmetric bilinear
form $-\tilde B|_{\tilde\frakp\times\tilde\frakp}$\,.
Let $\Ad$ and $\ad$ denote
the adjoint representations of $\Iso(\tilde N)$ and $\fraki(\tilde N)$\,, respectively.
Then $\Ad$ induces a faithful orthogonal representation of
$K$ on $\frakp$\,, by restriction; the corresponding infinitesimal action is
given by $\ad_{\tilde\frakp}:\frakk\to\so(\frakp),X\mapsto \ad(X)|_{\frakp}$\,.

\begin{definition}[{see~\cite[Example~3.7]{BCO}}]\label{de:symmetric_R-space}
In the situation described above, suppose that there exists some
$X\in\tilde\frakp$ with $\ad(X)^3=-\ad(X)$ and $X\neq 0$\,. Then $\Ad(\tilde K)\,X\subset\tilde\frakp$ is called a
(standard embedded) {\em irreducible symmetric R-space}.
\end{definition}

It is well known that every irreducible symmetric R-space $\Ad(\tilde K)\,X$ is a full symmetric submanifold of $\tilde\frakp$\,.
In particular, $\Ad(\tilde K)\,X$ is a parallel submanifold which is intrinsically a symmetric space; however,
$\Ad(\tilde K)\,X$ is not necessarily a (locally) {\em irreducible} symmetric space; this can be seen from Example~\ref{ex:extrinsic_holonomy2}. Conversely, as a consequence of~\cite[Theorem~3.7.8]{BCO} we have:

\begin{theorem}[Ferus]\label{th:Ferus}
For every full complete (intrinsically) locally irreducible parallel
submanifold $\tilde M$ of a Euclidian space $E$ there exists a simply connected irreducible
symmetric space $\tilde N$ of compact type which admits a
standard embedded irreducible symmetric R-space $\Ad(\tilde K)\,X$ and (after scaling the metric on $V$ by a positive factor) an
isometry $F:E\to\tilde\frakp$ such that $F(\tilde M)=\Ad(\tilde K)\,X$\,.
\end{theorem}
Clearly, for every irreducible symmetric R-space $\tilde M:=\Ad(\tilde K)\,X\subset\tilde\frakp$ we obtain the decomposition $\tilde\frakp=\tilde\frakp_-\oplus\tilde\frakp_+:=T_X\tilde M\oplus \bot_X\tilde M$ and the induced decomposition $\so(\tilde\frakp)=\so(\tilde\frakp)_+\oplus \so(\tilde\frakp)_-$ (see Eqs.~\eqref{eq:even} and~\eqref{eq:odd}).
Furthermore, let
$\ad(Z)_{\tilde\frakp}:=\ad(Z)|_{\tilde\frakp}:\tilde\frakp\to\tilde\frakp$ denote the induced
endomorphism of $\tilde\frakp$ for each $Z\in\tilde\frakk$\,.
In this situation, we have:

\begin{lemma}\label{le:10000}
Let $\Ad(\tilde K)\,X\subset\tilde\frakp$ be an irreducible symmetric R-space.
\begin{enumerate}
\item
There exist a decomposition
\begin{align}
\label{eq:Cartan_for_k}
&\tilde\frakk=\tilde\frakk_+\oplus\tilde\frakk_-\;\ \text{such that}\\
\label{eq:ad_frakp}
&\ad_{\tilde\frakp}(\tilde\frakk_+)\subset\so(\tilde\frakp)_+\qmq{and}\ad_{\tilde\frakp}(\tilde\frakk_-)\subset\so(\tilde\frakp)_-\,.
\end{align}
\item We have $\ad_{\tilde\frakp}(\tilde\frakk_-)=\tilde\fetth(T_X\tilde M)$ and $\ad_{\tilde\frakp}(\tilde\frakk_+)=[\tilde\fetth(T_X\tilde M),\tilde\fetth(T_X\tilde M)]$ (where the latter denotes the corresponding commutator ideal).
\item Let $\tilde h$ denote the second fundamental form of $\tilde M$ at $X$ and $\tilde \fetth:\tilde\frakp_-\to\so(\tilde\frakp)_-$ be the 1-form associated therewith according to~\eqref{eq:fetth1}. If some $A\in \so(\tilde\frakp)$
satisfies $[A,\tilde\fetth(\tilde\frakp_-)]\subset\tilde\fetth(\tilde\frakp_-)$\,, then $A\in
\ad_{\tilde\frakp}(\tilde\frakk)$\,.
\end{enumerate}
\end{lemma}
\begin{proof}
For Part~(a) see~\cite[Propositions A.2 and B.2~(a)]{J1}. For Parts~(b) and~(c) use \cite[Propositions B.2~(c) and~(d)]{J1}, respectively.
\end{proof}
Let $\tilde\frakc$ denote the center of $\tilde\frakk$\,.

\begin{corollary}\label{co:symmetric_R-space}
Let $\Ad(\tilde K)\,X\subset\tilde\frakp$ be an irreducible symmetric R-space, $\tilde h$ be its second fundamental form at $X$ and $\tilde \fetth:\tilde\frakp_-\to\so(\tilde\frakp)_-$ be the corresponding 1-form (see~\eqref{eq:fetth1}). If some $A\in\so(\tilde\frakp)$
satisfies $[A,\tilde\fetth(T_X\tilde M)]=\{0\}$\,, then there exists some $X\in\tilde\frakc$ with $\ad_{\tilde\frakp}(X)=A$\,.
If additionally $A\in\so(\tilde\frakp)_-$\,, then $X\in\tilde\frakc\cap\tilde\frakk_-$\,.
\end{corollary}
\begin{proof}
Using Lemma~\ref{le:10000}~(a) and~(b), we see that any such $A$ necessarily commutes with $\ad_{\tilde\frakp}(X)$ even for all $X\in\tilde\frakk$\,.
Furthermore, there exists some $Y\in\frakk$ with $A=\ad_{\tilde\frakp}(Y)$ according to Lemma~\ref{le:10000}~(c). Thus $Y\in\tilde\frakc$\,, since $\ad_{\tilde\frakp}$ is a faithful representation (in fact, $\ad_{\tilde\frakp}$ is equivalent to the linearized isotropy representation (cf.~\cite[Lemma B.1]{J1})). The last assertion follows from~\eqref{eq:ad_frakp}.
\end{proof}

\begin{lemma}\label{le:symmetric_R-space}
Suppose that $\tilde M\subset\tilde \frakp$ is a standard embedded irreducible symmetric R-space which is (intrinsically) a locally irreducible symmetric space. Then $\tilde\frakc\cap\tilde\frakk_-=\{0\}$\,.
\end{lemma}

\begin{proof}
Let $\tilde\frako$ denote the orthogonal complement of $\tilde\frakc\cap\tilde\frakk_-$ in
$\tilde\frakk_-$ with respect to $\tilde B$\,, and consider the
following two orthogonal symmetric Lie algebras:
\begin{itemize}
\item $\tilde\frakk_0\oplus\tilde\frako$ (which corresponds to a
simply connected symmetric space of compact type whose dimension is equal to the
dimension of $\tilde\frako$), and
\item $\tilde\frakc\cap\tilde\frakk_-$ (which corresponds to a Euclidian space
  whose dimension is equal to the dimension of $\tilde\frakc\cap\tilde\frakk_-$)
\end{itemize}
This is a decomposition of the orthogonal symmetric Lie algebra $\tilde\frakk$
(see~\eqref{eq:Cartan_for_k}) into ideals as described in~\cite[Ch.~V,~Theorem~1.1]{He};
hence, the universal covering space of $\tilde M$ splits off a
Euclidian factor whose dimension is equal to the dimension of
$\tilde\frakc\cap\tilde\frakk_-$ as described in the proof
of~\cite[Ch.~V,~Proposition~4.2]{He} (i.e. we apply the ``de\,Rham decomposition
theorem'' to the universal covering space of $\tilde M$). Therefore, $\tilde\frakc\cap\tilde\frakk_-=\{0\}$\,, since
$\tilde M$ is locally irreducible.
\qed \end{proof}

\paragraph{Proof of Proposition~\ref{p:parallel_flat_2}} By means of Lemma~\ref{le:Le1000}, it suffices to show that $\frakc(\tilde\fetth(T_pM))\cap\so(V)_-=\{0\}$ holds where $\tilde\fetth$ is the 1-form which is associated with the second fundamental form of a full intrinsically irreducible complete parallel submanifold of some Euclidian space $V$\,. Furthermore, by the strength of Theorem~\ref{th:Ferus}, we can assume that $V=\tilde\frakp$ and $\tilde M$ is a standard embedded irreducible symmetric R-space. Now Proposition~\ref{p:parallel_flat_2} is a consequence of Corollary~\ref{co:symmetric_R-space} in combination with Lemma~\ref{le:symmetric_R-space}.
\qed

\subsection{The case that $\bar M$ is locally irreducible}
\label{se:proof_of_theorem_2}
Given a parallel isometric immersion $f:M\to N$ into some symmetric space,
the second osculating bundle $\osc f$ is a parallel vector subbundle of $f^*TN$ (see Proposition~\ref{p:properties});
hence $\nabla^N$ defines a linear connection on $\osc f$ (through restriction).

\begin{definition}\label{de:HolonomyLieAlgebra}
Given a fixed point $p\in M$\,, let $\hol(\osc f)\subset\so(\osc_pf)$ denote the
holonomy Lie algebra of $\osc f$ with respect to $\nabla^N$ and the base point
$p$\,, called the ``extrinsic'' holonomy Lie algebra of $\osc f$ (see~\cite{J1}, Section~5).
\end{definition}

Furthermore, recall the definition of the linear map  $\fetth:T_pM\to\so(T_oN)$ (see~\ref{eq:fetth1}).
In this section, we will prove the following theorem:

\begin{theorem}\label{th:fetth_in_hol}
Let a parallel isometric immersion $f$
from a simply connected irreducible symmetric space $M$ with $\dim(M)\geq 3$ into the symmetric space $N$ be given.
If the totally geodesic submanifold $\bar M$ (see~\eqref{eq:tg}) is locally irreducible, then
\begin{equation}\label{eq:fetth_in_hol}
\fetth(T_pM)\subset\hol(\osc f)\;.
\end{equation}
\end{theorem}

 There exist higher-dimensional, extrinsically symmetric submanifolds
 of some irreducible symmetric space $N$ which are not locally irreducible and for which~\eqref{eq:fetth_in_hol} fails:

\begin{example}\label{ex:extrinsic_holonomy2}
Let $N$ be a simply connected irreducible symmetric space of compact type,
$\fraki(N)=\frakk\oplus\frakp$ be the Cartan
decomposition and suppose that there exists some $X\in\frakp$ with $\ad(X)^3=-\ad(X)$\,.
Then the orbit $M_X:=\Ad(K)\, X$ is an irreducible symmetric R-space (see Section~\ref{se:Theorem4}).
By means of the canonical identification $\frakp\cong T_pN$\,, we define a subset of $N$ via
$\tilde M_t:=\exp^N(t\,M_X)$ for each $t\in\R$\,.
Then the family $\tilde M_t$ where $t$ ranges over $\R$ is called a ``chain of spindles of
$M$'' (see~\cite{Q})\,; it is well known that in this situation $M_t$
is a symmetric submanifold of $N$ (see~\cite[Lemma~5]{Q}) which degenerates to a point for $t\in \Z\,\pi$
and which is totally geodesic for $t\in \frac{1}{2}\,\Z\,\pi$\,. These submanifolds
have been extensively studied by various authors before; in fact, they play the most
prominent role among the symmetric submanifolds of irreducible compact symmetric
spaces of higher rank (see~\cite[Proposition~9.3.3 and Theorem 9.3.4]{BCO}).

Now suppose that $N$ is an irreducible Hermitian symmetric space; for instance, there is the pair
$(\Sp(n)/\rmU(n),\rmU(n)/\SO(n))$\,. Keeping some $t\in]0,\pi/2[$ fixed, set $M:=M_t$
and let $\tau:\hat M\to M$ denote the universal covering; then
$f:=\iota^M\circ\tau:\hat M\to N$ is a 1-full\footnote{\label{Fn:1-full}This
  means that $\osc_qf=T_qN$ holds for all $q\in M$\,.} parallel isometric
immersion which does not satisfy~\eqref{eq:fetth_in_hol}, for the following reson: Let $\j$ denote the complex structure of
$T_pN$\,; then $\j\in\fetth(T_oM)$ holds but $\j$ does not belong to $\hol(f^*TN)$ (cf.~\cite[Theorem~1.9]{J1}).
In view of Theorem~\ref{th:MainResult}, this is because $M$ splits of locally a 1-dimensional factor: In fact, it follows from the proof of Lemma~\ref{le:symmetric_R-space} that $M_X$ splits of locally a 1-dimensional factor (since $\frakc\cap\so(\frakp)_-\neq\{0\}$\,, see~\cite[Lemma 6.12]{J1}); furthermore, $M_X\to M, Y\mapsto \exp^N(t\, Y)$ defines a
homothety (by a factor of $\sin(t)$) for each $t\in[0,\pi[$ \footnote{This map is injective for each $t\in[0,\pi[$ since $N$ is simply connected (see~\cite{Cr}).} and hence $M$ splits of locally a 1-dimensional factor, too.
Note that the three assertions of Theorem~\ref{th:MainResult} are nonetheless valid
here. A similar result holds also in the non-compact case (see~\cite[Proposition~9.3.8]{BCO}).
\end{example}

We aim to prove Theorem~\ref{th:fetth_in_hol}.
Let $f:M\to N$ be a parallel isometric immersion and $p$ be some fixed point of $M$\,. In the following, $\ad$ and $\Ad$ will denote the adjoint representations of
$\so(\osc_pf)$ and\ $\SO(\osc_pf)$\,, respectively. As a consequence of the
Jacobi identity, the linear map $\ad(A)$ is
a derivation of the Lie algebra $\so(\osc_pf)$ for each $A\in\so(\osc_pf)$\,, i.e. for all
$B,C\in\so(\osc_pf)$ we have
\begin{equation}
\label{eq:derivivation}
\ad(A)\,[B,C]=[\ad(A)\,B,C]+[B,\ad(A)\,C]\;.
\end{equation}
The next Proposition prepares a purely algebraic approach towards~\eqref{eq:fetth_in_hol}\,.

\begin{proposition}\label{p:results}
Let $f$ be a parallel isometric immersion from a simply connected
symmetric space $M$ into the symmetric space $N$\,, let $p$ be some fixed point of $M$ and consider the Lie algebra $\hol(\osc f)$ (see Definition~\ref{de:HolonomyLieAlgebra}).
\begin{enumerate}
\item
Let	 $\sigma^\bot:\osc_pf\to\osc_pf$ denote the linear reflection in $\bot^1_pf$\,.
Then we have
\begin{equation}\label{eq:Sigma(hol(osc_M))}
\Ad(\sigma^\bot)(\hol(\osc f))=\hol(\osc f)\;.
\end{equation}
Consequently, we obtain the decomposition
\begin{equation}\label{eq:splitting_of_hol}
\hol(\osc f)=\hol(\osc f)_+\oplus \hol(\osc f)_-
\end{equation}
with $\hol(\osc f)_+:=\so(\osc_pf)_+\cap\hol(\osc f)$ and $\hol(\osc f)_-:=\so(\osc_pf)_-\cap\hol(\osc f)$\;.
\item For each $x\in T_pM$\,, $\ad\big(\fetth(x)\big)$ defines an {\em outer derivation}\/ of
  $\hol(\osc f)$\,, i.e. we have
\begin{equation}\label{eq:outer_derivation}
[\fetth(x),\hol(\osc f)]\subset\hol(\osc f)\;.\footnote{Note that in
  case~\eqref{eq:fetth_in_hol} holds,~\eqref{eq:outer_derivation} is obvious
  (since $\hol(\osc f)$ is a Lie subalgebra of $\so(\osc_pf)$).}
\end{equation}
\item The vector space
\begin{equation}\label{eq:frakh}
\frakh:=\Spann{R^N(x,y)}{x,y\in T_pM}
\end{equation}
is a Lie subalgebra of $\so(T_{f(p)}N)$\,. For each
$A\in\frakh$ we have $A( T_pM)\subset T_pM$\,,
$A(\bot^1_pf)\subset\bot^1_pf$ and moreover $A|_{\osc_pf}\in\hol(\osc f)_+$\,.
\end{enumerate}
\end{proposition}
\begin{proof}
For Parts (a) and~(b) see~\cite[Theorem~5.2]{J1}\,.

For~(c): The fact that
$\frakh$ is a Lie subalgebra of $\so(T_{f(p)}N)$ follows from the curvature invariance of $T_pM$
(see Proposition~\ref{p:properties}~(a)) combined with the well known relation $R^N\cdot R^N=0$\,, i.e.
\begin{align*}
[R^N(u_1,u_2),R^N(v_1,v_2)]=R^N(R^N(u_1,v_2)\,v_1,v_2)+R^N(v_1,R^N(u_1,u_2)\,v_2)
\end{align*}
for all $u_1,u_2,v_1,v_2\in T_{f(p)}N$\,.
Furthermore, $R^N(x,y)(\osc_pf)\subset\osc_pf$ and $R^N(x,y)|_{\osc_pf}$ is
the curvature endomorphism of $\osc f$ at $p$\,, according to
Proposition~\ref{p:properties}~(f); thus $A|_{\osc_pf}\in\hol(\osc f)$ for
each $A\in\frakh$\,, by virtue of the ``Theorem of
Ambrose/Singer''\;. Since moreover $R^N(x,y)(T_pM)\subset T_pM$\,, we have $A|_{\osc_pf}\in\so(\osc_pf_+)$ for
each $A\in\frakh$\,, in accordance with Part~(a) of Lemma~\ref{le:splitting}. The result follows.
\qed \end{proof}

With the intent to show that the ``outer derivations'' mentioned in Part~(b) of Proposition~\ref{p:results}
are in fact ``inner derivations'' of $\hol(\osc f)$\,, we consider the usual
positive definite scalar product
on $\so(\osc_pf)$ given by
$$\g{A}{B}:=-\trace(A\circ B)\;.$$
It satisfies for all $A,B,C\in\so(\osc_pf)$
\begin{equation}\label{eq:ad_is_skew}
\g{[A,B]}{C}=\g{A}{[B,C]}\;.
\end{equation}
In other words, $\ad(A)$ is skew-symmetric for each $A\in\so(\osc_pf)$\,.

\begin{example}\label{ex:Sigma}
If $\sigma^\bot$ denotes the linear reflection in $\bot^1_pf$ (which is an
orthogonal map)\,,
then $\Ad(\sigma^\bot):\so(\osc_pf)\to \so(\osc_pf)$  is an orthogonal
map, too.
\end{example}

\begin{definition}\label{de:hat_fetth}
In the situation of Proposition~\ref{p:results}, let $P:\so(\osc_pf)\to\hol(\osc f)$ denote the orthogonal projection onto
$\hol(\osc f)$ with respect to the metric introduced above.
\end{definition}

Furthermore, let $\frakc(\hol(\osc f))$ denote the centralizer of $\hol(\osc
f)$ in $\so(\osc_p f)$ (see~\eqref{eq:second_subspace}).

\begin{lemma}\label{le:lambda}
In the situation of Proposition~\ref{p:results}, the following is true:
\begin{enumerate}
\item The outer derivation of $\hol(\osc f)$ induced by
$\ad\big(\fetth(x)\big)$ for each $x\in T_pM$ is actually an inner derivation of $\hol(\osc f)$\,; more
precisely, $[\fetth(x),A]=[P(\fetth(x)),A]$ holds for all $A\in\hol(\osc
M)$\,, i.e.
\begin{equation}\label{eq:707}
\fetth(x)-P(\fetth(x))\in\frakc(\hol(\osc f))\;.
\end{equation}
\item  We have
\begin{equation}\label{eq:801}
P(\fetth(x))\in\hol(\osc f)_-\;.
\end{equation}
\item The linear map $\fetth-P\circ\fetth:T_pM\to\so(\osc_pf)$ is injective or
  identically equal to $0$\,.
\end{enumerate}
\end{lemma}

\begin{proof}
For~(a):~\eqref{eq:707} is seen as follows: We can write $\fetth(x)=P(\fetth(x))+\fetth(x)^\bot$ with
$\fetth(x)^\bot\in\hol(\osc f)^\bot$\,. For each $A\in\hol(\osc f)$ we have:
\begin{equation}\label{eq:fetth(x)_2}
\underbrace{[\fetth(x),A]}_{\in\hol(\osc f)}=\underbrace{[P(\fetth(x)),A]}_{\in\hol(\osc f)}+[\fetth(x)^\bot,A]\;,
\end{equation}
from which we see that $[\fetth(x)^\bot,A]\in\hol(\osc f)$\,. We claim that
$[\fetth(x)^\bot,A]=0$ (and therefore~\eqref{eq:fetth(x)_2}
yields~\eqref{eq:707}):

In fact, we have
\begin{equation*}
\g{B}{[\fetth(x)^\bot,A]}=-\g{B}{[A,\fetth(x)^\bot]}\stackrel{\eqref{eq:ad_is_skew}}
=-\g{\underbrace{[B,A]}_{\in\hol(\osc f)}}{\fetth(x)^\bot}=0
\end{equation*}
for each $B\in\hol(\osc f)$\,, which implies that $[\fetth(x)^\bot,A]=0$\,, since $\gg$ is non-degenerate.

For~(b):
From~\eqref{eq:Sigma(hol(osc_M))} and Example~\ref{ex:Sigma} we
conclude that
$$\Ad(\sigma^\bot)|_{\hol(\osc f)}\circ P=P\circ\Ad(\sigma^\bot)\;,$$ hence
$\Ad(\sigma^\bot)\,P(\fetth(x))=P(\Ad(\sigma^\bot)\,\fetth(x))=-P(\fetth(x))$\,,
since $\fetth(x)\in\so(\osc_pf)_-$ (see~\eqref{eq:odd},~\eqref{eq:fetth1}); in this way~\eqref{eq:801} has been proved.

For~(c):
We have
$$\Kern(\fetth-P\circ\fetth)=\Menge{x\in T_pM}{\fetth(x)\in\hol(\osc f)}\,,$$
and therefore, since $T_pM$ is an irreducible $\hol(M)$-module,
it suffices to show that r.h.s.\ of the last equation is invariant under the
natural action of $\hol(M)$ on $T_pM$\,:

Let $y\in T_pM$ with
$\fetth(y)\in\hol(\osc f)$ and $A\in\hol(M)$ be given.
Thereby, according to~\eqref{eq:hol(M)}, without loss of generality we can assume
that there exist $x_1,x_2\in T_pM$ with $A=R^M(x_1,x_2)$\,; then
$$
\fetth(R^M(x_1,x_2)\,y)\stackrel{\eqref{eq:Gauss_Ricci},\eqref{eq:semiparallel_2}}{=}
[R^N(x_1,x_2),\fetth(y)]
-[[\fetth(x_1),\fetth(x_2)],\fetth(y)]\;.
$$
The second term of the r.h.s.\ of the last line is contained in $\hol(\osc f)$\,,
in accordance with~\eqref{eq:outer_derivation}. For the first term of the r.h.s.\ above,
note that
$
R^N(x_1,x_2)|_{\osc_pf}\in\hol(\osc f)$ holds according to Part~(c) of
Proposition~\ref{p:results}\,; the result now follows.
\qed \end{proof}

\begin{proposition}\label{p:Prop1001}
Let a parallel isometric immersion $f$
from a simply connected irreducible symmetric space $M$ into the symmetric space $N$ be given.
If the totally geodesic submanifold $\bar M$
(see~\eqref{eq:tg}) is locally irreducible, then
$\dim\big(\frakc(\hol(\osc f))\cap\so(\osc f)_-\big)\leq 2$\,.
\end{proposition}

It is not clear whether this estimate is optimal; at least the condition $\frakc(\hol(\osc f))\cap\so(\osc f)_-=\{0\}$ is
not always given (even if $M$ is irreducible):

\begin{example}
Let $N:=\C\rmP^n$ denote the complex projective space and
$p:=(1:0:\cdots:0)$\,. Then $T_pN=\C^n$ and the holonomy Lie algebra of $N$ is given by
$\fraku(n)=\R\,J\oplus\su(n)$ (where $J$ denotes multiplication by $\sqrt{-1}$)\,.
Let $M\subset N$ be a full, Lagrangian symmetric submanifold through $p$ and put $f:=\iota^M$\,.
The corresponding examples are listed in~\cite[Table 9.2]{BCO}
(for instance, there is an isometric embedding of $\SU(n)/\SO(n)$
onto a full, Lagrangian symmetric submanifold of $\C\rmP^N$ with $N=\frac{1}{2}(n-1)(n+2)$).
Then $M$ is always irreducible, $\bar M$ is a real projective space and $f$ is necessarily 1-full (cf. Footnote~\ref{Fn:1-full}).
In this situation, $\hol(f^*TN)$ is given by $\su(n)$ (see~\cite[Theorem~1.9]{J1}). Furthermore, since $M$ is
Lagrangian, we can assume w.l.o.g.\ that $T_pM=\R^n$ in which case $J\in\frakc(\su(n))\cap\so(2n)_-$\,.
\end{example}

\paragraph{Proof of Theorem~\ref{th:fetth_in_hol}}
Let a parallel isometric immersion $f$
from a simply connected irreducible symmetric space $M$ into the symmetric space $N$ be given.
Then $(\fetth-P\circ\fetth)(T_pM)\subset\frakc(\hol(\osc f))\cap\so(\osc f)_-$\,, according to Lemma~\ref{le:lambda} (a) and (b)\,.
If additionally the totally geodesic submanifold $\bar M$
(see~\eqref{eq:tg}) is irreducible, then $\dim(\frakc(\hol(\osc f))\cap\so(\osc f)_-)\leq 2$
according to Proposition~\ref{p:Prop1001}; hence, in case also $\dim(M)\geq 3$ holds, the linear map $\fetth-P\circ\fetth$ can
not be injective; thus $\fetth-P\circ\fetth$ vanishes identically, consequently to Part~(c) of Lemma~\ref{le:lambda}.
This implies that $\fetth(x)=P(\fetth(x))\in\hol(\osc f)$ holds for each $x\in T_pM$\,.
\qed

\paragraph{It remains to give the proof of Proposition~\ref{p:Prop1001}.}
For this, we will need the following concepts from representation theory:
Let a Lie algebra $\frakh$\,, Euclidian vector spaces $W$ and $U$ and orthogonal representations $\rho_1:\frakh\to
\so(W)$ and $\rho_2:\frakh\to \so(U)$ be given.

\begin{definition}\label{de:Hom} Put
\begin{equation}\label{eq:Hom}
\Hom_{\frakh}(W,U):=\Menge{\lambda\in\Hom(W,U)}{\forall\;h\in
  \frakh:\lambda\circ \rho_1(h)=\rho_2(h)\circ\lambda}\;.
\end{equation}
\end{definition}

\begin{lemma}\label{le:Hom1}
In the above situation, suppose that $W$ is an irreducible
$\frakh$-module. Then:
\begin{enumerate}
\item Each $\lambda\in\Hom_\frakh(W,U)$	 is either an
injective map or identically equal to $0$\,; in case $\lambda\neq 0$ its image
$\lambda(W)$ is an irreducible $\frakh$-submodule of $U$ and
$\lambda^{-1}:\lambda(W)\to W$ is an $\frakh$-homomorphism, too.
\item If $\mu\in \Hom_\frakh(W,W)$ is self-adjoint, then there exists some
  $\kappa\in\R$ such that $\mu=\kappa\,\Id$\,.
\item
We have $\dim(U)\geq\dim(W)\cdot\dim(\Hom_\frakh(W,U))/d$\;.
\end{enumerate}
\end{lemma}
\begin{proof}
For Part~(a): This is usually known as ``Schur's Lemma''. Part~(b) follows from~(a)
(since $\mu$ has at least one real eigenvalue).
For Part~(c): Decompose $U$ into irreducible submodules and then use~(a).
\end{proof}

Now we consider the Lie algebra $\frakh$ defined by~\eqref{eq:frakh}
and its linear representations
\begin{align}\label{eq:rho_1}
&\rho_1:\frakh\to\so(T_pM),\;A\mapsto
A|_{T_pM}:T_pM\to T_pM\;,\\
\label{eq:rho_2}
&\rho_2:\frakh\to\so(\bot^1_pf),\;A\mapsto A|_{\bot^1_pf}:\bot^1_pf\to\bot^1_pf\;,
\end{align}
proposed by Part~(c) of Proposition~\ref{p:results}\,.

Let $\Hom_\frakh(T_pM,\bot^1_pf)$ and $\Hom_\frakh(T_pM,T_pM)$ be the linear spaces defined according to
Def.~\ref{de:Hom}. We define the integer
\begin{equation}\label{eq:d}
d:=\dim(\Hom_\frakh(T_pM,T_pM))\;.
\end{equation}

\begin{lemma}\label{le:My_lemma}
Let $f$ be a parallel isometric immersion from a simply connected, irreducible
symmetric space $M$ into $N$ such that the totally geodesic submanifold $\bar M$ defined by~\eqref{eq:tg} is locally irreducible.
\begin{enumerate}
\item $T_pM$ is an irreducible $\frakh$-module and
\begin{equation}\label{eq:hol(M_top)_1}
\hol(\bar M)=\rho_1(\frakh)\;.
\end{equation}
\item We have $d\leq 2$\,.
\item If $\dim(\bot^1_pf)>2$\,, then $\hol(\osc f)_-\neq \{0\}$\,.
\item The following map is injective,
\begin{equation}\label{eq:Hom_3}
\frakc(\hol(\osc f))\cap\so(\osc f)_-\hookrightarrow\Hom_\frakh(T_pM,\bot^1_pf), A\mapsto A|_{T_pM}\;.
\end{equation}
\end{enumerate}
\end{lemma}

\begin{proof}
Equation~\ref{eq:hol(M_top)_1} follows from Eqs.~\eqref{eq:hol(M_top)_2},~\eqref{eq:frakh}
and~\eqref{eq:rho_1}. Furthermore, $\hol(\bar M)$ acts irreducibly on $T_pM$ since $\bar M$ is locally irreducible (see Proposition~\ref{p:hol-relations})\,. Part~(a) follows.

For Part~(b): Let $\mu\in\Hom_\frakh(T_pM,T_pM)$ be given. We decompose
$\mu=\mu_1+\mu_2$ with $\mu_1$ self-adjoint and $\mu_2$ skew. Then $\mu_1$ and
$\mu_2$ both belong to $\Hom_\frakh(T_pM,T_pM)$\,; hence $\mu_1=\kappa_1\,\Id$ and
$\mu_2^2=-\kappa_2^2\,\Id$ holds for certain $\kappa_i\in\R$ by means of Schur's Lemma; thus, if $\mu_2\neq
0$\,, then $J:=1/\kappa_2\, \mu_2$ is a Hermitian structure of $T_pM$ which
commmutes with $\hol(\bar M)$ according to~\eqref{eq:hol(M_top)_1}, i.e. the universal covering space of
$\bar M$ is a Hermitian symmetric space and $\mu_2$ is a multiple of the Hermitian
structure of this symmetric space (recall that the Hermitian structure of an
irreducible Hermitian symmetric space is uniquely determined up to a sign according
to~\cite[Ch.~VIII, \S 7]{He}). Therefore, $\dim(\Hom_\frakh(T_pM,T_pM))\leq 2$\,.

For Part~(c):
By contradiction, assume that $\hol(\osc f)_-=\{0\}$\,. Using
Part~(a) of Proposition~\ref{p:results},~\eqref{eq:fetth1} and
the rules for $\Z_2$ graded Lie algebras, we conclude that
$$
\forall A\in\hol(\osc f)_+, x\in T_pM:[\fetth(x),A]=0\;.
$$
Let $A\in\frakh$ be given and put $A_1:=\rho_1(A)\in\so(T_pM)$ and $A_2:=\rho_2(A)\in\so(\bot^1_pf)$\,.
Consequently to Part~(c) of Proposition~\ref{p:results}, the endomorphism
$A|_{\osc_pf}=A_1\oplus A_2$
belongs to $\hol(\osc f)_+$\,. The previous implies that for all $x\in T_pM$
\begin{align*}
&[A,\fetth(x)]=0\,,\ \text{thus}\ \forall y\in T_pM:\,A_2\,\fetth(x)\,y=\fetth(x)\,A_1\,y\;.
\end{align*}
Hence, for all $x,y\in T_pM$
$$
h(x,A_1\,y)=A_2\,h(x,y)=A_2h(y,x)=h(y,A_1\,x)\;.
$$
Multiplication of the last equation with $\xi\in\bot_pf$ yields
$$
\g{x}{S_\xi A_1\,y}=\g{h(x,A_1\,y)}{\xi}=\g{h(y,A_1\,x)}{\xi}=\g{y}{S_\xi A_1\,x}\;.
$$
Since $A_1$ is skew-symmetric, whereas $S_\xi$ is symmetric, it follows that
$$
A_1\circ S_\xi=-S_\xi \circ A_1\;;
$$
and therefore
$$
\forall \xi,\eta\in\bot_pf: A_1\circ S_\xi \circ S_\eta=-S_\xi \circ A_1\circ S_\eta=S_\xi\circ S_\eta\circ A_1\;.
$$
We hence see: $S_\xi\circ S_\eta\in\Hom_\frakh(T_pM,T_pM)$ for all $\xi,\eta\in\bot_pf$\,.
Since $S_\xi^2$ is a self adjoint linear map for each $\xi\in\bot^1_pf$\,, the previous
implies that there exists some $\kappa\in\R$ with
$S_\xi^2=\kappa\cdot\Id$ (see Lemma~\ref{le:Hom1}~(b))\,; in
particular, $S_\xi$ is invertible for each $\xi\in\bot^1_pf$ unless $\xi=0$ and thus
$$
\bot^1_pf\to\Hom_{\frakh}(T_pM,T_pM),\;\eta\mapsto S_\xi\circ S_\eta
$$
is injective, which is implies that $\dim(\bot^1_pf)\leq d\leq 2$\,. This proves Part~(c).

For~(d):
Let $A\in\frakc(\hol(\osc f))\cap\so(\osc f)_-$ be given. Using Part~(c) of
Proposition~\ref{p:results}, it is straightforward to show that
$A|_{T_pM}:T_pM\to\bot^1_pf$ belongs to $\Hom_\frakh(T_pM,\bot^1_pf)$\,. Hence
the map described by~\eqref{eq:Hom_3} is well defined; injectivity follows
from Lemma~\ref{le:splitting}.
\qed \end{proof}

\paragraph{Proof of Proposition~\ref{p:Prop1001}}
Let a parallel isometric immersion $f$
from a simply connected irreducible symmetric space $M$ into the symmetric space $N$ be given such that
the totally geodesic submanifold $\bar M$ is irreducible. If $\dim(M)=1$\,, then the result is trivial (since then $\dim(\so(\osc_pf))=1$).
In the following, we assume that $\dim(M)\geq 2$\,. By contradiction, assume that there exist
three linearly independent elements
$A_1,A_2,A_3\in\frakc(\hol(\osc f))\cap\so(\osc f)_-$\,.
Then $\lambda_1:=A_1|_{T_pM},\lambda_2:=A_2|_{T_pM},\lambda_3:=A_3|_{T_pM}\in\Hom(T_pM,\bot^1_pf)$
are linearly independent elements of $\Hom_\frakh(T_pM,\bot^1_pf)$\,, consequently to Lemma~\ref{le:My_lemma}~(d)\,.
Put $U_j:=\lambda_j(T_pM)$\,, then $\lambda_j$ is an isomorphism onto $U_j$
according to Lemma~\ref{le:Hom1}. It is not possible that $U_1=U_2=U_3$\,, since otherwise
$\Id=\lambda_1^{-1}\circ\lambda_1$\,, $\lambda_1^{-1}\circ \lambda_2$\,,
$\lambda_1^{-1}\circ \lambda_3$ were three linearly independent elements of
the vector space $\Hom_\frakh(T_pM,T_pM)$\,, which is not possible because of Lemma~\ref{le:My_lemma}~(b).
Therefore, without loss of generality we may assume that $U_1\neq U_2$\,; then even $U_1\cap U_2=\{0\}$\,,
since $U_1$ and $U_2$ are irreducible $\frakh$-modules; in particular,
$\dim(\bot^1_pf)\geq 2\,\dim(M)\geq 4$\,.
We claim that this already implies that $\hol(\osc f)_-=\{0\}$\,:

By the above, the linear maps
$$
\lambda_j=A_j|_{T_pM}:T_pM\;\to\;U_j\qmq{and }\lambda_j^*=-A_j|_{U_j}:U_j\;\to\;T_pM
$$
are linear isomorphisms for $j=1,2$\,; therefore,
for every $x\in	 T_pM$ and $j=1,2$ there exists $\xi_j\in U_j$ with
$A_j(\xi_j)=x$\,. Furthermore, given $A\in \hol(\osc f)_-$\,, we have
$[A_1,A]=[A_2,A]=0$\,, according to~\eqref{eq:707}; thus
$$
Ax=A(A_j\xi_j)=A_j(A\,\xi_j)\in U_j \qmq{for $j=1,2$\,,}
$$
and therefore $A\,x\in U_1\cap U_2=\{0\}$\,. We obtain $A|_{T_pM}=0$ and because
of Lemma~\ref{le:splitting} even $A=0$\,, i.e. $\hol(\osc f)_-=\{0\}$\,.

Therefore, $\dim(\bot^1_pf)\leq 2$\,,
according to Lemma~\ref{le:My_lemma}~(c)\,. On the other hand,
we have already seen that $\dim(\bot^1_pf)\geq 4$\,, a contradiction.
\qed

\subsection{Proof of Theorem~\ref{th:MainResult}}
\label{se:proof}
Suppose that $N$ is of compact or non-compact type and
let a parallel isometric immersion $f:M\to N$ from a simply
connected, irreducible symmetric space $M$ of dimension at least three be
given.

Besides Assertions $(a)$\,, $(b)$ and $(c)$ from
Theorem~\ref{th:MainResult}, we also introduce the following two assertions, $(d)$ and $(e)$\,,
as follows:

\begin{enumerate}\addtocounter{enumi}{3}
\item Equation~\ref{eq:fetth_in_hol} holds.
\item The totally geodesic submanifold $\bar M$ (see~\eqref{eq:tg}) is locally irreducible.
\end{enumerate}
We will now prove the chain of implications ``$(e)\Rightarrow (d)\Rightarrow (c)\Rightarrow (b)\Rightarrow
(a)\Rightarrow (e)$''; which, in particular, gives the proof of Theorem~\ref{th:MainResult}.

\paragraph{For ``$(a)\Rightarrow (e)$'':}
By contradiction, assume that $\bar M$ is not locally irreducible. Then $\bar M$ is a
flat, according to Proposition~\ref{p:hol-relations}, hence $f(M)$ is
contained in some flat according to Theorem~\ref{th:4}.
\qed
\paragraph{For ``$(e)\Rightarrow (d)$'':}
See Theorem~\ref{th:fetth_in_hol}.
\paragraph{For ``$(d)\Rightarrow (c)$'':}
This direction is an immediate consequence of Theorem~\ref{th:homogen}
combined with the following lemma:

\begin{lemma}\label{le:frakk}
Let a parallel isometric immersion $f$ into the symmetric space $N$ be
given. Let $\hol(N)$\,, $\hol(f^*TN)$  and $\hol(\osc f)$ denote the holonomy Lie
algebras of $TN$\,, $TN|_M$ and $\osc f$ with respect to $\nabla^N$
and the base point $p$\,, respectively. Furthermore, let
$\rho_*:\frakk\to\so(T_pN)$ denote the linearized isotropy representation (see~\eqref{eq:rho_star}).
\begin{enumerate}
\item
There is a sequence of inclusions of Lie algebras,
\begin{equation}\label{eq:hol(N)_2}
\hol(f^*TN)\subset \hol(N)\subset\rho_*(\frakk)\;.
\end{equation}
\item
For each $A\in \hol(\osc f)$ there exists some $\tilde A\in\hol(TN|_M)$
with
\begin{equation}\label{eq:killing_1}
\tilde A(\osc_pf)\subset\osc_pf\qmq{and}\tilde A|_{\osc_pf}=A\;.
\end{equation}
\end{enumerate}
Hence~\eqref{eq:fetth_in_hol} implies Assertion~(b) of Theorem~\ref{th:homogen}\,.
\end{lemma}

\begin{proof}
For~(a): The first inclusion in~\eqref{eq:hol(N)_2} is trivial.
The second one can be seen as follows:

Given $A\in\hol(N)$\,, by $g_t:=\exp(t\,A)$ is defined a one-parameter subgroup of the
Holonomy group $\Hol(N)\subset\SO(T_{f(p)}N)$\,.
Thus we have $g_t=\ghdisp{0}{1}{c_t}{N}$ for some loop $c_t:[0,1]\to N$ with
$c_t(0)=p$ for each $t\in \R$\,.
Since the curvature tensor of $N$ is parallel, we have
\begin{equation}\label{eq:derive_2}
\forall u,v\in T_{f(p)}N, t\in\R:g_t\circ R^N(u,v)\circ g_t^{-1}=R^N(g_t\,u,g_t\,v)\,.
\end{equation}
Thus there exists an isometry $G_t$ of $N$ with $G_t(p)=p$ and $T_pG_t=g_t$\,, as a consequence of the ``Theorem of Cartan/Ambrose/Hicks''. The result follows.

For~(b): Remember that $\osc f\subset f^*TN$ is a $\nabla^N$-parallel vector
subbundle, according to Proposition~\ref{p:properties}~(f)\,. Therefore,
using an argument on the level of the corresponding Holonomy groups, we
conclude that for each $A\in\hol(f^*TN)$ we have $A(\osc_pf)\subset \osc_pf$\,,
$A|_{\osc_pf}\in\hol(\osc f)$ and the canonical map $\hol(f^*TN)\to\hol(\osc
f),\;A\mapsto A|_{\osc_pf}$ is surjective. Furthermore, $\hol(f^*TN)\subset\hol(N)$ is a Lie subalgebra. The result
now follows from~(a).
\qed \end{proof}

\paragraph{For ``$(c)\Rightarrow (b)$'':}
This direction is trivial.
\paragraph{For ``$(b)\Rightarrow (a)$'':}
If $f(M)$ is a homogeneous submanifold, then $f:M\to f(M)$ is a Riemannian
covering (cf.\ the proof of Theorem~\ref{th:homogen}, $(a)\Rightarrow (b)$). Hence $(b)\Rightarrow (a)$	 is an
immediate consequence of the following result:

\begin{proposition}\label{p:flat}
If $N$ is a symmetric space of compact or non-compact type and $M\subset N$
is a homogeneous submanifold which is contained
in some flat of $N$\,, then $M$ is a flat, too. In particular, then the
universal covering space of $M$ is a Euclidian space.
\end{proposition}

\begin{proof}
Let $\tilde M$ be a flat of $N$ with $p\in\tilde M$\,. Let $\Gamma:\frakp\to
T_pN$ be the inverse of the canonical isomorphism $\frakp\to T_pN, X\mapsto X^*_p$\,.
Since $\tilde M$ is a totally
geodesic submanifold of $N$\,, it is well known that then
$\frakm:=\Gamma(T_p\tilde M)$ is a ``Lie triple system'', i.e. $[\frakm,[\frakm,\frakm]]\subset\frakm$\,;
(cf.~\cite[Ch.~IV,~\S~7]{He}). Set
$$\fraki(\tilde M,N):=\Menge{X\in\fraki(N)}{X^*(\tilde M)\subset T\tilde M}\;.$$
We claim that $\fraki(\tilde M,N)=[\frakm,\frakm]\oplus\frakm$ holds:

We have $\frakm\subset\fraki(\tilde M,N)$\,, since $\tilde M$ is totally
geodesic; hence $[\frakm,\frakm]\oplus\frakm\subset\fraki(\tilde M,N)$
since $\fraki(\tilde M,N)$ is a Lie algebra (in fact, the latter is the Lie algebra
of the group of isometries of $N$ which leave $\tilde M$ invariant).
In the other direction, let $[\frakm,\frakm]^\bot$ denote the orthogonal complement of
$[\frakm,\frakm]\oplus\frakm$ in $\fraki(\tilde M,N)$ with respect to the
Killing form $B$ of $\fraki(N)$\,. It suffices to show that $[\frakm,\frakm]^\bot=\{0\}$\,.
To this end, one knows that $\Gamma$ is an equivariant map
of $\frakk$-modules, i.e.
\begin{equation}\label{eq:correspondence}
\forall X\in\frakk,y\in T_pN:\ad(X)\,\Gamma(y)=\Gamma(\rho_*(X)\,y)\;.
\end{equation}
Furthermore, we have $\rho_*(X)\, T_p\tilde M\subset T_p\tilde M$ for every $X\in \fraki(\tilde M,N)$\,;
thus $\ad(\fraki(\tilde M,N))\,\frakm\subset\frakm$ by means of~\eqref{eq:correspondence}\,.
Furthermore, $0=B(X,[Y,Z])=B([X,Y],Z)$ for all $Y,Z\in\frakm$ and $X\in [\frakm,\frakm]^\bot$\,; since $B|_{\frakm\times\frakm}$ is negative or
positive definite, we now conclude from the previous that $\ad(X)\,\frakm=\{0\}$ holds for all
$X\in[\frakm,\frakm]^\bot$\,, i.e. (see~\eqref{eq:correspondence}) $\rho_*(X)=0$\,.
Therefore,	$[\frakm,\frakm]^\bot=\{0\}$ because $\rho_*$ is faithful.

Moreover, since $\tilde M$ is a
flat of $N$\,, we have $[\frakm,\frakm]=\{0\}$ (cf.~\cite[Ch.~V,~Proposition~6.1]{He}), hence $\fraki(\tilde M,N)=\frakm$\,.
Suppose now that $M$ is a homogeneous submanifold of $N$ which passes through
$p$\,, say $M=G\,p$ for some subgroup
$G\subset\Iso(N)^0$\,, and that there exists a flat $\tilde M$ of $N$ with $M\subset
\tilde M$\,. Since the connected components of
the intersection of any two flats of $N$ are flats of $N$\,, too, without loss of generality we may assume that $\tilde M$ is
the connected component of the intersection of all the flats of $N$ which contain $M$\,.
Then $g(\tilde M)=\tilde M$ for each $g\in G$\,, and thus
$\frakg\subset\fraki(\tilde M,N)$ (where $\frakg$ denotes the Lie algebra of $G$)\,.
Hence, by the previous, $\frakg\subset\frakm\subset\frakp$\,, thus $M=\exp^N(T_pM)$\,; the result follows.
\qed \end{proof}

\section{2-symmetric submanifolds}
\label{se:two-symmetric}
\begin{definition}[{see~\cite[Ch.~7.2]{BCO}}]\label{de:2-symmetric}
A submanifold $M\subset N$ will be called 2-symmetric, if $M$ is a symmetric
space (whose geodesic symmetries are denoted by $\sigma^M_p$)
and for every $p\in M$ there exists an isometric involution $I_p$ of $N$ such that
\begin{equation}
I_p(M)=M\qmq{and} I_p|_M=\sigma^M_p
\end{equation}
\end{definition}

\begin{example}
Every symmetric submanifold of $N$ is a 2-symmetric parallel submanifold.
\end{example}

\begin{definition}\label{de:inner_type}
Let $M$ be a symmetric space. We say that $M$ is of inner
  type if $\Iso(M)^0$ contains the geodesic symmetries of $M$\,.
\end{definition}
For a complete list of the irreducible symmetric spaces which are of inner type see Proposition~\ref{p:2-symmetric}.

\begin{theorem}\label{th:two-symmetric}
Suppose that $N$ is of compact or non-compact type and let a full parallel isometric immersion $f:M\to N$ from a
simply connected, irreducible symmetric space $M$ with $\dim(M)\geq 3$ be given.
If $M$ is additionally of inner type, then $\tilde M:=f(M)$ is a 2-symmetric
submanifold.
\end{theorem}
\begin{proof}
Since $f:M\to N$ is a full parallel isometric
immersion, $\tilde M:=f(M)$ is, in particular, not contained in any flat of $N$\,.
Thus, according to Theorem~\ref{th:MainResult}, $\tilde M$ is a submanifold
with extrinsically homogeneous	holonomy bundle and $f:M\to
\tilde M$ is a Riemannian covering.
Let $G$ be a subgroup of $\Iso(N)$ which has the properties described in
Definition~\ref{de:homogeneous_holonomy}. We claim that already $\tilde M$ is a
symmetric space of inner type:

For every smooth geodesic line
$\gamma:\R\to M$ let $\Theta_\gamma(t)$ denote the 1-parameter subgroup of
transvections along $\gamma$ (see~\eqref{eq:Theta}
and~\eqref{eq:Tr}). Let $q\in M$\,, set $p:=f(q)$\ and let $\sigma^M_q$
denote the geodesic symmetry at $q$\,; then $\sigma^M_q\in\Iso(M)^0$ since $M$
is of inner type.  Thus, using Proposition~\ref{p:isometrygroup},
we see that there exist certain smooth geodesic lines $\gamma_i:[0,1]\to M$ ($i=1,\ldots,n$) such that
\begin{equation}\label{eq:222}
\sigma^M_q=\Theta_{\gamma_1}(1)\circ\cdots\circ\Theta_{\gamma_n}(1)\;.
\end{equation}
Furthermore, $f\circ\gamma_i$ is a curve into $\tilde M$; hence, in
accordance with Definition~\ref{de:homogeneous_holonomy}, for each $i=1,\ldots,n$ there
exists $g_i\in G$ with $g_i(\tilde M)=\tilde M$\,, $g_i(f(\gamma_i(0)))=f(\gamma_i(1))$ and
\begin{equation}\label{eq:333}
Tg_i|_{T_{f(\gamma_i(0))}\tilde M}=\ghdisp{0}{1}{f\circ \gamma_i}{\tilde M}\;.
\end{equation}
Then, since $f:M\to \tilde M$ is  a Riemannian
covering,~\eqref{eq:Tr},~\eqref{eq:222} and~\eqref{eq:333} imply that
$g_i\circ f=f\circ \Theta_{\gamma_i}$ for all ($i=1,\ldots,n$).
Put $g:=g_1\circ\cdots\circ g_n$\,; then $T_pg|_{T_p\tilde M}=-\Id$
holds, i.e. $g|_{\tilde M}$ is the geodesic symmetry
$\sigma^{\tilde M}_p$\,. Thus $\tilde M$ is a symmetric space of inner type (since $G$ is connected).

Furthermore, we have just seen that for each $p\in \tilde M$ there
exists some involution $g\in G$ with $g|_{\tilde
M}=\sigma^{\tilde M}_p$\,. But $g^2|_{\tilde M}=\Id$ implies that $g^2=\Id$
according to Lemma~\ref{le:unique}~(c) (since $\tilde M$ is a full submanifold), i.e. $\tilde M$ is 2-symmetric. The result follows.
\qed
\end{proof}

\section{Acknowledgement}
I want to thank the members of the Mathematical Institute of Cologne for their support.
A special heartly thanks to my teacher, Professor Helmut~Reckziegel, for his helpful
suggestions. I also want to thank Professor Gudlaugur~Thorbergsson, who advised me to some useful literature.

\section*{Appendix}
\begin{appendix}
\section{The canonical connection}
\label{se:A}
Let $M$ be an arbitrary Riemannian manifold, $G$ be a connected Lie group and
$G\times M\to M$ be a transitive isometric action, i.e. $M$ is a Riemannian homogeneous $G$-space.
For an arbitrary point $p\in M$ let $H\subset G$ be the corresponding isotropy
subgroup; then $M\cong G/H$\,. Let $\frakg$	 and $\frakh$ denote
the Lie algebras of $G$ and $H$\,, respectively. Recall that a {\em reductive
  complement}\/ of $\frakh$ is a subspace $\frakm\subset\frakg$
such that $\frakg=\frakh\oplus\frakm$ and $\Ad(h)(\frakm)\subset\frakm$ for all $h\in H$\;,
where $\Ad$ denotes the adjoint representation of $G$
(cf.~\cite[A.~3]{BCO}). Then $\frakh\oplus\frakm$ is also called a {\em reductive
decomposition}\/ of $\frakg$\,.

\begin{definition}\label{de:homogeneous}
Let $M\cong G/H$ be a Riemannian homogeneous $G$-space.
A vector bundle $\bbE\to M$ is called a
\emph{homogeneous vector bundle} if there exists an action
$\alpha:G\times\bbE\to\bbE$ through vector bundle isomorphisms such that the
bundle projection of $\bbE$ is equivariant.
\end{definition}

If $\bbE$ is a homogeneous vector bundle over $M$ and $\frakm$ is a
reductive complement, then there exists a distinguished linear connection $\nabla^\bbE$ on $\bbE$\,, called the
\emph{canonical connection}\,. In the framework of~\cite{KN}, it can be obtained as follows:
\begin{equation}\label{eq:tau}
\tau:G\to M\;, g\mapsto g(p)\;,
\end{equation}
is a principal fiber bundle,
\begin{equation}\label{eq:scrH}
\scrH_g:=\Menge{X_g}{X\in\frakm}
\end{equation}
defines a $G$-invariant connection $\scrH$ on this principle bundle
(where the elements of $\frakm$ are also considered as left-invariant
vector fields on $G$\,, see~\cite[Vol.~1, p.~239]{KN}).
Since $\bbE$ is a vector bundle associated with $\tau$ via
\begin{equation}\label{eq:assoziierung}
G\times \bbE_p\to\bbE\;,(g,v)\mapsto\alpha(g,v)\;,
\end{equation}
the connection $\scrH$ induces a linear connection $\nabla^{c}$ on $\bbE$\,,
see~\cite[Vol.~1, p.~87]{KN}.
One can show that $\nabla^c$ does not depend on the special choice of the base
point $p$\,; therefore it is called the canonical connection.
In order to relate	the parallel displacement in $\bbE$ induced by
$\nabla^c$ to the horizontal structure $\scrH$\,, let a curve
$c:\R\to M$ with $c(0)=p$ be given; then
\begin{equation}\label{eq:pardisp}
\forall v\in \bbE_p:\ghdisp{0}{1}{c}{\nabla^c}\,v=\alpha(\hat c(1),v)\;,
\end{equation}
where $\hat c:[0,1]\to G$ denotes the $\scrH$-lift of $c$ with $\hat c(0)=\Id$\,.

\begin{example}\label{ex:geodesics}
Let $M\cong G/H$ be a Riemannian homogeneous $G$-space and $\frakg=\frakh\oplus\frakm$ be a reductive decomposition.
\begin{enumerate}
\item  For each $X\in\frakm$ the
  1-parameter subgroup $:\R\to G,\,t\mapsto\exp(t\,X)$ is the integral curve of $X$ and
  hence, in accordance with~\eqref{eq:scrH}, this is the horizontal lift
  of the curve $c$ defined by $c(t):=\exp(t\,X)(p)$\,.
Therefore, in view of~\eqref{eq:pardisp}, if $\bbE\to M$ is a homogeneous
vector bundle, then we have for all $X\in\frakm, v\in\bbE_p$\,:
\begin{equation}\label{eq:frakm333}
t\mapsto   \alpha(\exp(t\,X),v)\ \ \text{is a $\nabla^c$-parallel section of
  $\bbE$ along $c$}\;.
\end{equation}
\item
The induced action $\alpha^M:G\times TM\to TM$ equips $\bbE:=TM$ with the structure of a
homogeneous vector bundle over $M$\,; let $\nabla^c$ denote the corresponding canonical connection on $TM$\,.
In accordance with~\cite[Ch.~X, Corollary~2.5]{KN}, the $\nabla^c$-geodesics
$\gamma:\R\to M$
with $\gamma(0)=p$ are given by
\begin{equation}\label{eq:pardisp2}
\gamma(t)=\exp(t\,X)(p)\ \ \text{($X\in\frakm$)\,.}
\end{equation}
\end{enumerate}
\end{example}

\begin{proposition}[{see~\cite[Ch.~X, Theorem~2.8]{KN}}]\label{p:KN}
Let $M$ be a Riemannian manifold and $\nabla^M$ be the Levi Civita connection.
Suppose that there exists a Lie group $G$ and some $p\in M$ such that
\begin{itemize}
\item $G$ acts effectively on $M$ through $\nabla^{M}$-parallel vector bundle
isomorphisms,
\item and for every curve $c:[0,1]\to M$ with $c(0)=p$ there exists some $g\in G$ with
\begin{equation}
\forall y\in T_pM:\ghdisp{0}{1}{c}{M}\,y=\alpha(g,y)\;.
\end{equation}
\end{itemize}
Then $M\cong G/H$ is a Riemannian homogeneous $G$-space and there exists a unique reductive decomposition $\frakg=\frakh\oplus\frakm$
such that $\nabla^{M}$ is the corresponding canonical connection on $TM$ (as described in Example~\ref{ex:geodesics})\,.
\end{proposition}

Now suppose that $M$ is a symmetric space. Clearly, $\Iso(M)^0\times M\to M$
is a transitive action and the Cartan decomposition $\fraki(M)=\frakk\oplus\frakp$ with respect to
our base point $p$ is a reductive decomposition.

\begin{proposition}[{see~\cite[Ch.~XI]{KN}}]
\label{p:B}
For every symmetric space $M$\,, the tangent bundle is a homogeneous vector bundle
and the Levi Civita connection is the canonical connection induced by the Cartan decomposition.
\end{proposition}

Returning to the extrinsic situation, let $N$ be a symmetric space, $p\in N$\,, and $G\subset \Iso(N)^0$ be a
connected subgroup. Then the orbit $M:=G\,p$ is a homogeneous
$G$-space and the pullback bundle $TN|_M$ is a homogeneous vector bundle
over $M$ via the induced action $\alpha:G\times TN|_M\to TN|_M$ (see
Def.~\ref{de:homogeneous}). Suppose that there exists a reductive complement $\frakm\subset\frakg$
and let $\nabla^c$ be the corresponding canonical connection on $TN|_M$\,. The direct sum
$\nabla^M\oplus\nabla^\bot$ defines a second connection on $TN|_M=TM\oplus\bot
M$\,; we set $\Delta:=\nabla^M\oplus\nabla^\bot-\nabla^c\in\Hom(TM,\End(TN|_M))$\,.
Similar as in~\cite[Lemma~2.2]{Co}, in this situation we have:

\begin{lemma}\label{le:parallel}
Let $N$ be a symmetric space, $p\in N$ and $G\subset \Iso(N)^0$ be a
connected subgroup. Consider the orbit $M:=G\,p$\,.
\begin{enumerate}
\item
$TM$\,, $\bot M$\,, the first normal bundle $\bot^1M$ and the second osculating bundle $\osc
M=TM\oplus\bot^1M$ are $\nabla^c$-parallel vector subbundles of $TN|_M$\,.
\item
Both $\Delta$ and $h$ are parallel sections of
$\Hom(TM,\End(TN|_M))$ and $\Sym^2(TM,\bot M)$ with respect
to $\nabla^c$\,.
\end{enumerate}
\end{lemma}
\begin{proof}
For~(a): We have
$Tg(TM)\subset TM$\,, $Tg(\bot M)\subset \bot M$ and
$h\big(Tg(x),Tg(y)\big)=Tg\big(h(x,y)\big)$ (since $G$ is a subgroup of
$\Iso(N)$ with $g(M)=M$ for each $g\in G$), hence the vector bundles listed in~(a) are
invariant under the action of $G$\,.
Thus it suffices to show that every $G$-invariant vector subbundle
$\bbF\subset TN|_M$ is parallel with respect to
$\nabla^c$ and that the corresponding connection on $\bbF$ (obtained by
restriction) is the canonical connection induced by the action $\alpha|_{G\times \bbF}:G\times\bbF\to\bbF$\,:

Let $c:[0,1]\to M$ be a curve with $c(0)=p$\,, set $q:=c(1)$ and let
$v\in \bbF_p$ be given. Then
there exists some $g\in G$ with $g(p)=q$\,. Let $\hat
c:[0,1]\to G$ be the $\scrH$-lift of $c$ with $\hat c(0)=g$ (see~\eqref{eq:scrH}),
then according to~\eqref{eq:pardisp}

$$\ghdisp{0}{1}{c}{\nabla^c}\, v=\alpha(\hat c(1),v)\in\bbF_q\;,$$
because of the $G$-invariance of $\bbF$\,; hence $\bbF$ is parallel along
$c$\,. The result now follows, since the canonical connection is independent of
the base point $p$\,.

For~(b): Using similar arguments
as before, $G$-invariance of $h$ and $\Delta$ implies $\nabla^c$-parallelity,
respectively. Hence it suffices to show that $h$ and $\Delta$ both are $G$-invariant:

Let us first verify the statement for $h$\,. Because $G$ is a
subgroup of $\Iso(N)$\,, we have for each $g\in G$\,:
$$
\forall x,y\in T_pM:h(T_pg\,x,T_pg\,y)=T_pg\,h^f(x,y)\;.
$$
This implies that $h$ is $G$-invariant.

To see that also $\Delta$ is $G$-invariant, note that $G$ acts through
vector bundle isomorphisms on $TN|_M$ which are parallel with respect to both $\nabla^c$ (by construction of
the canonical connection, see~\eqref{eq:pardisp})
and $\nabla^M\oplus\nabla^\bot$ (because $G$ is a subgroup of
the isometries of $N$). Being the difference of two $G$-invariant linear
connections, $\Delta$ is $G$-invariant, too.
\qed \end{proof}

\section{An alternate description of the Lie bracket}
\label{se:B}
Let $M$ be an arbitrary Riemannian manifold, $\Iso(M)$ be the corresponding Lie group of isometries and
$\fraki(M)$ be its Lie algebra. Let $K$ denote the isotropy subgroup in
$\Iso(M)^0$ with respect to some base point $p$\,, $\rho:K\to\SO(T_pN)$
be the isotropy representation and $\rho_*:\frakk\to\so(T_oM)$ be the linearized isotropy
representation (see~\eqref{eq:rho_star}). Our first aim is to extend $\rho_*$
to a map from $\fraki(M)$ in a natural way. To this end, for each
$X\in\fraki(M)$ we consider the one-parameter subgroup of $\Iso(M)$ which is
given by $\psi_t^X:=\exp(t\,X)$ and we introduce
the corresponding ``fundamental vector
field'' $X^*$ on $M$ (in the sense of~\cite{KN}) defined by
\begin{equation}\label{eq:fundamental_vector_field}
\forall p\in N:X^*(p):=\frac{\diff}{\diff t}\Big|_{t=0}\psi_t^X(p)\;;
\end{equation}
then $\psi_t^X$ ($t\in\R$) is the flow of $X^*$\,. Consider the covariant
derivative $\nabla^MX^*$ for each $X\in\fraki(M)$\,, which is a skew-symmetric
tensor field of type $(1,1)$ on $M$\,; then we introduce
\begin{align}
\label{eq:pi_1}
&\pi_1:\fraki(M)\to\so(T_pM)\,,\;X\mapsto\nabla^MX^*(p)\;,\\
&\pi_2:\fraki(M)\to T_pM\,,\; X\mapsto X^*(p)\;.
\label{eq:pi_2}
\end{align}

\begin{proposition}[{\cite[Vol.~1, p.~245]{KN}}]\label{p:Kovv_X}
Given a Riemannian manifold $M$\,, let $\fraki(M)$ and $\frakk$ be defined as above. We have
\begin{equation}\label{eq:Kovv_X}
\forall X\in\fraki(M),u\in T_pM:\nabla^MX^*\,u=\frac{\nabla^M}{dt}\Big|_{t=0}T_p\psi_t^X\,u\;.
\end{equation}
In particular, $\pi_1|_\frakk=\rho_*$ holds.
\end{proposition}

By means of Eqs.~\eqref{eq:pi_1} and~\eqref{eq:pi_2}, we define a ``bracket'' on the linear space
$\rho_*(\frakk)\oplus T_pM$\,,
\begin{align}\label{eq:bracket_1}
&\forall x,y\in T_pM:[x,y]:=-R^M(x,y)\;,\\
\label{eq:bracket_2}
&\forall A\in\rho_*(\frakk),x\in T_pM:[A,x]:=-[x,A]:=A\,x\;,\\
\label{eq:bracket_3}
&\forall A,B\in\rho_*(\frakk):[A,B]:=A\circ B-B\circ A\;.
\end{align}

\begin{lemma}\label{le:bracket}
Let $M$ be a Riemannian manifold and let $\fraki(M)$\,, $\frakk$ and $\rho_*$ be defined as before.
Equations~\ref{eq:bracket_1} -~\ref{eq:bracket_3} equip $\rho_*(\frakk)\oplus T_pM$ with
the structure of a Lie algebra such that the linear map
\begin{equation}\label{eq:iota}
\iota:\fraki(M)\to \rho_*(\frakk)\oplus T_pM\,,\; X\mapsto(\pi_1(X),\pi_2(X))
\end{equation}
becomes an isomorphism of $\Z_2$-graded Lie algebras. In particular, the following two equations
hold:
\begin{align}
\label{eq:frakk_1}
&\frakk=\Kern(\pi_2)\,,\\
&\frakp=\Kern(\pi_1)\;.
\label{eq:frakp_1}
\end{align}
\end{lemma}
\begin{proof}
Equation~\ref{eq:frakk_1} is straightforward.
For~\eqref{eq:frakp_1}: (see also~\cite[Theorem~2.2.20]{Klg})
L.h.s.\ is contained in r.h.s.\ as a consequence of
Proposition~\ref{p:B} combined with Example~\ref{ex:geodesics} and~\eqref{eq:Kovv_X}; equality now follows the faithfulness of
$\rho_*$\,. The Injectivity of $\iota$ now follows;
surjectivity of $\iota$ is an immediate consequence of Proposition~\ref{p:Kovv_X}.
For the bracket relations cf.~\cite[Ch.~XI]{KN}.
\qed \end{proof}

\section{On the isometry group of a symmetric space}
\label{se:C}
Let $M$ be a symmetric space whose geodesic symmetries at the various points
$p\in M$ are denoted by $\sigma_p$ and whose Cartan decomposition is given
by $\fraki(M)=\frakk\oplus\frakp$\,.

For every smooth geodesic line $\gamma$ of $M$ with $\gamma(0)=p$ we have the family of
{\em transvections}\/ along $\gamma$\,,
given by
\begin{equation}
\forall t\in\R:\Theta_\gamma(t):=\sigma_{\gamma(t/2)}\circ\sigma_{\gamma(p)}\;.
	\label{eq:Theta}
\end{equation}
According to~\cite[Lemma~2.2.5]{Klg}, we have
\begin{align}\label{eq:Tr}
&\Theta_\gamma(t)(\gamma(p))=\gamma(t)\qmq{and}
T_p\Theta_\gamma(t)=\ghdisp{0}{t}{\gamma}{M}\;;
\end{align}
in particular, $\Theta_\gamma(t)$ is a differentiable one-parameter subgroup of $\Iso(M)$\,.
Let $\Tr(M)$ denote the subgroup of $\Iso(M)$ which is generated by the transvections.

\begin{proposition}\label{p:isometrygroup}
Let $M$ be a locally irreducible irreducible symmetric space. Then
$\Iso(M)^0=\Tr(M)$ holds.
\end{proposition}

\begin{proof}
As a consequence of~\eqref{eq:Tr}, every element of $G:=\Tr(M)$ can be
joined with $\Id$ by a $C^\infty$-path in $\Iso(M)$\,; thus it follows from a
result of Freudenthal (see~\cite[Vol.~1, p.~275]{KN}) that $G$ is already a connected Lie subgroup of $\Iso(M)$\,.
Let $\frakg$ denote the Lie algebra of $G$\,; we claim that
$\frakp\subset\frakg$ holds: Since $\nabla^M$ is the canonical connection of
$TM$ according to Proposition~\ref{p:B}, 
for each $X\in\frakp$ the curve $\gamma(t):=\exp(t\,X)(p)$ is a geodesic of $M$ and
$T_p\exp(t\,X)\,y$ is a parallel section of $TM$ along $\gamma$ for each $y\in
T_pM$ (see Example~\ref{ex:geodesics}).
Thus $\exp(t\,X)=\Theta_\gamma(t)$ for all $t\in\R$\,, as a consequence
of~\eqref{eq:Tr}\,, i.e. $\frakp\subset\frakg$\,. Furthermore, we claim that $\frakk=[\frakp,\frakp]$ holds:

In fact, by the local irreducibility of $M$\,, the Killing form $B$ of $\fraki(M)$ is
positive or negative; hence $B([X,Y],Z)=B(X,[Y,Z])=0$
for some $Z\in\frakk$ and all $X,Y\in\frakp$ if and only if $Z=0$ (cf.\ the proof of Proposition~\ref{p:flat}).

Therefore, we even have $\frakg=\fraki(M)$\,.
\qed\end{proof}

\begin{proposition}\label{p:2-symmetric}
$M$ is of inner type (see Definition~\ref{de:inner_type}) if and only if $M$ is isometric to
\begin{itemize}
\item the compact irreducible Hermitian symmetric spaces, i.e.\\
  $\SO(n+2)/\SO(2)\times \SO(n)$ ($n\geq 1$), $\SU(m+n)/\rmS(\rmU(m)\times\rmU(n))$
  ($m,n\geq 1$)\,, $\SO(2n)/\rmU(n)$\,,
  $\Sp(n)/\rmU(n)$ ($n\geq 1$)\,, $\rmE_6/\rmT\cdot\mathrm{Spin}(10)$\,, $\rmE_7/\rmT\cdot\rmE_6$\,,
\item$\SO(m+n)/\SO(m)\times \SO(n)$ ($1\leq m,n$ where $m$ or $n$ is even).
\item $\Sp(m+n)/\Sp(m)\times\Sp(n))$ ($m,n\geq 1$)\,,
\item the following exceptional symmetric spaces of compact type,
$\rmE_8/\SO(16)$\,, $\rmE_8/\rmE_7\times\SU(2)$\,,
 $\rmF_4/\Sp(3)\times SU(2)$\,, $\rmF_4/ \mathrm{Spin}(9)$\,,
 $\rmG_2/\SO(4)$\,,
\item or the non-compact duals of the symmetric spaces listed above.
\end{itemize}
\end{proposition}

\begin{proof}
We recall the following result (see~\cite[Ch.~IX,~Corollary~5.8]{He}):
In case $M$ is of non-compact type, $\Iso^0(M)$ contains the geodesic
reflections of $M$ if and only if $\frakk$ contains a maximal Abelian
subalgebra of $\fraki(M)$\,.

Furthermore, for every simply connected, irreducible symmetric space $M$ of
compact type we have the (non-compact) ``dual symmetric space'' $M^*$
(see~\cite[Ch.~V, \S2]{He}) and we carefully verify
that we may translate the above result one to one over to $M^*$\,.
Recalling the classification of symmetric
spaces (cf.~\cite[A.~4, Tables A.~1-A.~5]{BCO})), we now see that in Theorem~\ref{th:two-symmetric}
there are listed exactly the simply connected, irreducible symmetric spaces $M$ for
which $\frakk$ contains a maximal Abelian subalgebra of $\fraki(M)$\,. The result now follows.
\qed \end{proof}
\end{appendix}

\vspace{2cm}
\begin{center}
 \qquad
 \parbox{60mm}{Tillmann Jentsch\\
  Tachenbergstrasse 41\\
  70499 Stuttgart\\
  Germany\\[1mm]
  \texttt{tilljentsch@web.de}}
\end{center}
\end{document}